\documentclass[openany, amssymb, psamsfonts]{amsart}
\usepackage{mathrsfs,comment}
\usepackage{amsmath}
 \usepackage{graphicx}
\usepackage[usenames,dvipsnames]{color}
\usepackage[normalem]{ulem}
\usepackage{url}
\usepackage[all,arc,2cell]{xy}
\UseAllTwocells
\usepackage{enumerate}
\usepackage{footnote, footmisc}
\usepackage{hyperref}  
\hypersetup{%
  bookmarksnumbered=true,%
  bookmarks=true,%
  colorlinks=true,%
  linkcolor=blue,%
  citecolor=blue,%
  filecolor=blue,%
  menucolor=blue,%
  pagecolor=blue,%
  urlcolor=blue,%
  pdfnewwindow=true,%
  pdfstartview=FitBH}

\def\makeautorefname#1#2{\expandafter\def\csname#1autorefname\endcsname{#2}}
\def\equationautorefname~#1\null{(#1)\null}
\makeautorefname{footnote}{footnote}%
\makeautorefname{item}{item}%
\makeautorefname{figure}{Figure}%
\makeautorefname{table}{Table}%
\makeautorefname{part}{Part}%
\makeautorefname{appendix}{Appendix}%
\makeautorefname{chapter}{Chapter}%
\makeautorefname{section}{Section}%
\makeautorefname{subsection}{Section}%
\makeautorefname{subsubsection}{Section}%
\makeautorefname{theorem}{Theorem}%
\makeautorefname{thm}{Theorem}%
\makeautorefname{cor}{Corollary}%
\makeautorefname{lem}{Lemma}%
\makeautorefname{prop}{Proposition}%
\makeautorefname{pro}{Property}
\makeautorefname{conj}{Conjecture}%
\makeautorefname{defn}{Definition}%
\makeautorefname{notn}{Notation}
\makeautorefname{notns}{Notations}
\makeautorefname{rem}{Remark}%
\makeautorefname{quest}{Question}%
\makeautorefname{exmp}{Example}%
\makeautorefname{ax}{Axiom}%
\makeautorefname{claim}{Claim}%
\makeautorefname{ass}{Assumption}%
\makeautorefname{asss}{Assumptions}%
\makeautorefname{con}{Construction}%
\makeautorefname{prob}{Problem}%
\makeautorefname{warn}{Warning}%
\makeautorefname{obs}{Observation}%
\makeautorefname{conv}{Convention}%

\newcommand{\Aut}{\text{Aut}}

\newcommand{\Z}{\mathbb{Z}}
\newcommand{\N}{\mathbb{N}}

\newtheorem{thm}{Theorem}[section]
\newtheorem{cor}{Corollary}[section]
\newtheorem{prop}{Proposition}[section]
\newtheorem{lem}{Lemma}[section]

\theoremstyle{definition}
\newtheorem{defn}{Definition}[section]

\newtheorem{exmp}{Example}[section]

\newtheorem{quest}{Question}[section]

\newtheorem{obs}{Observation}[section]
\newtheorem{conv}{Convention}[section]

\makeatletter
\let\c@obs=\c@thm
\let\c@cor=\c@thm
\let\c@prop=\c@thm
\let\c@lem=\c@thm
\let\c@prob=\c@thm
\let\c@con=\c@thm
\let\c@conj=\c@thm
\let\c@defn=\c@thm
\let\c@notn=\c@thm
\let\c@notns=\c@thm
\let\c@exmp=\c@thm
\let\c@ax=\c@thm
\let\c@pro=\c@thm
\let\c@ass=\c@thm
\let\c@warn=\c@thm
\let\c@rem=\c@thm
\let\c@sch=\c@thm
\let\c@equation\c@thm
\numberwithin{equation}{section}
\makeatother

\title{On the Existence of Partition of the Hypercube Graph Into 3 Initial Segments}

\author{Ethan Soloway, Megan Triplett, Wenshi Zhao}

\thanks{This paper was supported in part by NSF DMS grant No. 2349684. The work for this paper was conducted during 2024 Auburn University Summer REU}

\begin{document}

\maketitle

\begin{abstract}

Let $Q_n = \{0, 1\}^n$ be a hypercube graph. The initial segment $I_k \subseteq Q_n$ is the subset consisting of the first $k$ vertices of $Q_n$ in the binary order. A pair of integers $(a, b) \in \Z_{>0}^2$ is said to be fit if, whenever $2^n \geq a+b$, there exists $g_1, g_2 \in \Aut(Q_n)$ such that $g_1(I_a) \cup g_2(I_b) = I_{a+b}$, and $(a,b)$ is unfit otherwise. For $a + b + c = 2^n$, there is a partition of $Q_n$ into $3$ initial segments of length $a, b$, and $c$ if and only if $(a, b)$ is a fit pair. Thus, the notion of fit and unfit pairs is closely related to the graph-partition problem for hypercube graphs. This paper introduces a new criterion in determining whether $(a,b)$ is fit using an easy-to-compute point-counting function and applies this criterion to generate the set of all unfit pairs. It further shows that the number of unfit pairs $(a,b)$, where $0 < a,b \leq 2^n$, is $4^n - \binom{4}{1}3^n + \binom{4}{2} 2^n - \binom{4}{1}$, which is also the number of surjection of an $n$-element set to a $4$-element set.
\end{abstract}

\tableofcontents

\section{Introduction}

\begin{defn}[Hypercube Graph]
\label{hypercube}
Let $Q_n :=\{0, 1\}^n$ denote the set of binary strings of length $n$. We give $Q_n$ the structure of a graph by saying that two strings $x, y \in \{0, 1\}^n$ share an edge iff they differ in exactly one digit.
\end{defn}

\begin{defn}[Binary Order]
\label{binorder}
There is a canonical ordering on $Q_n$. Given $x, y \in Q_n$, we view them as binary representations of natural numbers and use the standard $x \leq y$ ordering on the naturals. The \textbf{initial segment} of length $a$, denoted $I_a \subseteq Q_n$, is the set of $x \in Q_n$ such that $x < a$. Note this is only well defined for $a \leq 2^n$.
\end{defn}

\begin{exmp}
\label{initexample}
Consider $I_6 \subseteq Q_3$. This is the set $I_6 = \{ 000, 001, 010, 011, 100, 101\}$ as shown in Figure \ref{sixinit}.
\end{exmp}

\begin{figure}[h] \label{sixinit}
\includegraphics[width=3cm]{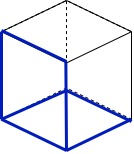}
\caption{ $I_6$ in the hypercube $Q_3$}
\end{figure}

\begin{conv}
\label{autconvention}
Let $Aut(Q_n)$ be the set of graph automorphisms of $Q_n$. For $S \subseteq Q_n$ and $g \in G$, we will let $g(S)$ denote the image of $S$ under the automorphism  $g$. For $S, T \subseteq Q_n$, we will write $S \cong T$ if their induced subgraphs are isomorphic.
\end{conv}

\begin{defn}[Edge Boundary]
\label{edgeboundary}
Let $G$ be a finite simple graph and $S \subseteq V$ be an arbitrary subset of vertices. Then $\delta S = \{ (u, v) \in E | u \in S, v \notin S\}$, i.e. the set of edges with exactly one vertex in $S$. If we view $S$ as partitioning $V$ into $(S, V \setminus S)$, this can also be interpreted as the edges crossing the partition.
\end{defn}

\begin{quest}[Edge Isoperimetric Problem]
\label{eip}
Fix some positive integer $k$. Which sets $S \subseteq G$ with $|S| = k$ minimize the quantity $|\delta S|?$
\end{quest}

\noindent In particular, people asked this question for the hypercube graph $G = Q_n$. The following theorem is originally proven in \cite{HarperProof}, \cite{LindseyProof}, \cite{BernsteinProof}, and \cite{HartProof}.

\begin{thm}[Harper, Lindsey, Bernstein, Hart] \label{hypercubeeip} For $S \subseteq Q_n$ non-empty, $|S|= k$, $S$ minimizes the edge boundary iff $S \cong I_k$.
\end{thm}

\noindent Since the edge isoperimetric problem can also be viewed as minimizing the edges crossing a partition of $Q_n$ into $2$ sets, it is natural to consider partioning $Q_n$ into $3$, or in general $k$ sets. 

\begin{defn}[Partition Boundary]
\label{partbound}Let $A = \{A_i\}_{i = 1}^k$ be a partition of $Q_n$ into $k$ sets. Then $\delta A := \{(u, v) \in E(Q_n) | u \in A_i, v \in A_j, i \neq j\}$.
\end{defn}

\noindent Previous results have concerned partioning $Q_n$ into $k$ approximately equal sets, i.e. differing in cardinality by at most one (cite Bezrukov). However, in  this paper we consider 3-partioning the hypercube into sets of arbitrary size. \\

\noindent The following result, which appears on \cite{HarperBook} p.7, gives us a convenient language.

\begin{prop}
\label{transitiveaction}
Consider $I_a$ as an induced subgraph and let $i : I_a \to Q_n$ be an arbitrary embedding of the initial segment of length $a$ into $Q_n$. Then there is some $g \in Aut(Q_n)$ such that $i(I_a) = g(I_a)$.
\end{prop}

\begin{obs}
\label{3initpartition}
Let $A$ be a $3$-partition of $Q_n$ such that $A_1 \cong I_a$, $A_2 \cong I_b$, and $A_3 \cong I_c$, for some positive integers $a + b + c = 2^n$. Then $|\delta A| = \frac{1}{2}(|\delta A_1| + |\delta A_2| + |\delta A_3|)$ is minimized over all partitions of size $(a, b, c)$ by Theorem \ref{hypercubeeip}. It is natural to ask when $Q_n$ can be partitioned into such $A_1, A_2, A_3$. It is also equivalent to say that given $(a, b, c)$, there are $A, B$ such that $A \cong I_a$, $B \cong I_b$, and $A \cup B \cong I_{a + b}$, as $Q_n \setminus (A \cup B)$ must then be isomorphic to $I_c$. By Proposition \ref{transitiveaction}, we must have $A = g_1(I_a)$, $B = g_2(I_b)$ and $A \cup B = g_3(I_{a + b})$, for some $g_1, g_2, g_3 \in Aut(Q_n)$. Thus, we have $g_1(I_a) \cup g_2(I_b) = g_3(I_{a + b})$. Applying $g_3^{-1}$ to both sides this is equivalent to $g_1(I_a) \cup g_2(I_b) = I_{a + b}$ for some $g_1, g_2 \in Aut(Q_n)$.
\end{obs}

\begin{quest}[Essential Question]
\label{essentialquest} Given positive integers $a, b$ with $a + b \leq 2^n$, when does there exist $g_1, g_2 \in Aut(Q_n)$ such that $g_1(I_a) \cup g_2(I_b) = I_{a + b}$?
\end{quest}

\begin{defn}[Unfit Pairs]
\label{fitunfit} Given a pair of positive integers $(a, b)$ with $a, b \leq 2^n$, we say that $(a, b)$ is \textbf{fit} if it satisfies the condition in Question \ref{essentialquest}. Otherwise, $(a, b)$ is called \textbf{unfit}.
\end{defn}

\begin{obs}
\label{fitobs}
Note that if $(a, b)$ is a fit pair with $a + b + c = 2^n$, then $(b, a)$, $(a, c)$, $(c, a)$, $(b, c)$ and $(c, b)$ are also fit pairs. Therefore, our set of fit pairs displays a 6-fold symmetry.\\ \\
\end{obs}

\section{Basic Tools}
\begin{defn}[Exclusive-Or]
\label{xordef}
Let $x \in Q_n$. For $y \in Q_n$, let $x \oplus y$ denote the exclusive-or operation with inputs $x$ and $y$. Equivalently, if $x$ and $y$ are thought of as elements of $(\mathbb{Z}/2)^n$, this is just component-wise addition. Fixing some $x \in Q_n$, the map $y \mapsto x \oplus y$ is an automorphism of $Q_n$.
\end{defn}

\begin{defn}[Permutation]
\label{permutedef} 
Let $x \in Q_n$ and $\sigma$ be a permutation of the numbers$\{0, 1, ... n - 1\}$. If the binary representation for $x$ is $x_0x_1...x_{n - 1}$, then define $\sigma(x) = x_{\sigma(0)}x_{\sigma(1)}...x_{\sigma(n - 1)}$. For fixed $\sigma$, the map $y \mapsto \sigma(y)$ is again an automorphism of $Q_n$.
\end{defn}

\begin{thm}[Automorphism Decomposition]
\label{autdecomp}
Let $g \in Aut(Q_n)$. Then, there is a unique permutation $\sigma$ and unique $x \in Q_n$ such that $g(y) = x \oplus \sigma(y)$ for all $y \in Q_n$. 

\end{thm}

\begin{proof}
This is a translation of known results about the structure of $\Aut(Q_n)$ into elementary language.  Clearly, each $\sigma$ in the symmetric group of $n$ elements gives distinct action on $Q_n$, and each map $y \rightarrow x \oplus y$ gives distinct action on $Q_n$. If $x_1 \oplus \sigma_1$ = $x_2 \oplus \sigma_2$, then since the vertex $00\ldots0$ is sent to $x_1$ by the first map and $x_2$ by the second map, $x_1 = x_2$. Consequently, $\sigma_1 = \sigma_2$. Hence, the $x \oplus \sigma$'s give $2^n\cdot n!$ distinct automorphisms. The theorem follows from the fact that $|\Aut(Q_n)| = 2^n \cdot n!$ by \cite{AutQn}.
\end{proof}

\noindent We next introduce the point-counting function central to our discussion.

\begin{defn}[m-Vector]
Let $S \subseteq Q_n, 0 \leq i \leq n-1$. Define $m_i(S)$ to be the total number of 1's that appear on the $i$th position of elements in $S$. Define $m(S)$ to be the vector $(m_0(S), m_1(S), \ldots, m_{n-1}(S))$. 
\end{defn}

\begin{exmp}
$I_5, I_6 \subseteq Q_3$. $m(I_5) = (1, 2, 2),$ and $m(I_6) = (2, 2, 3)$.
\end{exmp}

\begin{lem}[Ordering Lemma]
\label{orderinglemma}
For $m(I_a) = (m_{0},m_{1},...,m_{n - 1}),$ $ m_{0} \leq m_{1} \leq ... \leq m_{n - 1}$. 
\end{lem}
\begin{proof}
By the construction of initial segments, the sequence of $0$s and $1$s in a column $i$ can be written as a sequence of $2^{n-i-1}$ $0$s followed by $2^{n-i-1}$ $1$s. Thus, the partial sum of the sequence in column $i+1$ must be greater than the corresponding partial sum of the sequence in column $i$. 
\end{proof}

\noindent The action of the automorphism on the m-vector can be easily described. 

\begin{prop}
\label{mvectoraction}
Let $S \subseteq Q_n$, and $m(S) = (m_0, m_1, \ldots, m_{n-1})$
\begin{enumerate}
\item Let $x = x_0 x_1 \ldots x_{n-1} \in Q_n$. Let $J$ denote the set of index $j$ such that $x_j = 1$. Then 
$$m_j(x \oplus S) = \left\{ \begin{array}{rcl}
|S| - m_j(S) & \mbox{if}
& j \in J \\ m_j(S) & \mbox{if} & j \not\in J 
\end{array}\right.$$

\item Let $\sigma \in S_n$. Then $m(\sigma(S)) = \sigma(m(S))$, where $\sigma$ permutes the $n$ elements in the $m-vector$.
\end{enumerate}
\end{prop}

\begin{proof}
Fix $j$. If $j \in J$, then $x \oplus S$ turns all $0$'s into $1$'s and all $1$'s into $0$'s. So the new number of $1$'s is simply the number of $0$'s before applying the automorphism. Thus, $m_j(x \oplus S) = |S| - m_j(S)$. The other parts of the proposition follows directly from definition.
\end{proof}

\noindent Note that for any $g \in \Aut Q_n$ and $S\subseteq Q_n$, $m(g(S))$ can be described using only the information of $m(S)$ and $|S|$. Thus, if $m(S_1)$ = $m(S_2)$ and $|S_1|=|S_2|$, then $m(g(S_1)) = m(g(S_2))$. \\

\noindent The following proposition shows an important property unique to the initial segments.

\begin{prop}
\label{uniqueness prop}
Let $I_a \subseteq Q_n$.
Let $S \subseteq Q_n$, $|S| = a$. Then $m(I_a) = m(S)$ if and only if $I_a = S$.
\end{prop}

\begin{proof} We use the idea of a weighted sum. For any subset $S \subseteq Q_n$, each element may be regarded as the binary representation of an integer. The weighted sum of $S$ is referred to as the sum of all such integers. By definition, this sum is also $2^{n-1}m_0(S) + 2^{n-2}m_1(S) + \cdots + 2^0 m_{n-1}(S)$. \\

\noindent If $I_a=S$, then $m(I_a)=m(S)$. Conversely, let $I_a \subseteq Q_n$. Suppose $S \subseteq Q_n$, $|S| = a$, and $m(I_a) = m(S)$. Observe that $I_a$ is the unique subset of $Q_n$ having the smallest weighted sum among subsets of size at least $a$ since $I_a$ represents the integers $\{0,1,\ldots, a-1\}$. Since $m(I_a) = m(S)$, $I_a$ and $S$ must have the same weighted sum, which implies that $I_a = S$.
\end{proof}

\begin{prop}
\label{left multiplication}
Let $S_1, S_2, \ldots, S_p, T_1, T_2, \ldots, T_q \subseteq Q_n$, where $|S_1| + \cdots + |S_p| = |T_1| + \cdots + |T_q|$. Let $g \in Aut(Q_n)$. Then $$m(S_1) + \cdots + m(S_p) = m(T_1) + \cdots + m(T_q)$$ if and only if $$m(g(S_1)) + \cdots + m(g(S_p)) = m(g(T_1)) + \cdots + m(g(T_q)).$$
\end{prop}

\begin{proof} Let $g = x \oplus \sigma \in Aut(Q_n)$. The image $m(g(S))$ is uniquely determined by $m(S)$ and $|S|$, so we may define a natural action $g(v, n)$ on the set of pairs $(v, n)$, $v \in \Z^n_{\geq 0}$ and $n \in \Z_{\geq 0}$ using the rules in Proposition \ref{mvectoraction}. Moreover, regarding this set as a semigroup under addition, we also have $g(v+w, n+m) = g(v,n) + g(w,m)$ again by Proposition \ref{mvectoraction}. \\

\noindent Suppose, $m(S_1) + \cdots + m(S_p) = m(T_1) + \cdots + m(T_q)$. Then
\begin{align*}
m(g(S_1)) + \cdots + m(g(S_p)) &= g(m(S_1), |S_1|) + \cdots + g(m(S_p), |S_p|) \\
&= g(m(S_1) + \cdots + m(S_p), |S_1| + \cdots + |S_p|) \\
&= g(m(T_1) + \cdots + m(T_q), |T_1| + \cdots + |T_q|) \\
&= g(m(T_1), |T_1|) + \cdots + g(m(T_q), |T_q|) \\
&= m(g(T_1)) + \cdots + m(g(T_q)).
\end{align*}
\noindent The "if" direction is directly obtained by applying $g^{-1}$ to each term and using the "only if" direction.
\end{proof}

\noindent Next, we introduce a visual representation of how the automorphism group acts on the m-vectors.

\begin{defn}[m-Table]
Let $S \subseteq Q_n$, $m(S) = (m_0, m_1, \ldots, m_{n-1})$. Then the m-table of $S$ is defined as 

$$\left( \begin{array}{cccc} m_0 & m_1 & \ldots & m_{n-1} \\
|S| - m_0 & |S| - m_1 & \ldots & |S| - m_{n-1} \end{array} \right)$$

\noindent The rows will be referred to as row $0$ and row $1$. The $j$th column will refer to the column that corresponds to $m_j$. The entry in $(i,j)$ denotes the entry in row $i$ and column $j$. 
\end{defn}

\noindent By the ordering lemma, if $(m_0, m_1, \ldots, m_{n-1})$ is the m-vector of $I_a$, then $m_0 \leq m_1 \leq \cdots \leq m_{n-1}$. Furthermore, $m_{n-1} \leq a/2$, which means that $m_{n-1} \leq a - m_{n-1}$. Again by the ordering lemma, $a - m_{n-1} \leq \cdots \leq a - m_1 \leq a - m_0$. This describes the total order on the entries of the m-table of $I_a$.

\begin{lem}
\label{m-table min}
Let $I_a \subseteq Q_n$, and let $\alpha$ be the smallest integer such that $m_{\alpha}(I_a) \neq 0$. Then the smallest non-zero number in the m-table of $I_a$ is $a - 2^{n-1-\alpha}$.
\end{lem}

\begin{proof}
Let $I_a \subseteq Q_n$, and let $\alpha$ be the smallest integer such that $m_{\alpha}(I_a) \neq 0$. Consider the representation of an intial segment $I_a$ as the set of $n$ sequences $\{M_0, \cdots M_n \}$, where $M_i$ is the sequence of $0$s and $1$s in a column $i$. Each sequence $M_i$ can be written as a sequence of $2^{n-i-1}$ $0$s followed by $2^{n-i-1}$ $1$s followed by $2^{n-i-1}$ $0$s $\cdots$ until the sequence contains $a$ terms. Because $\alpha$ is the smallest integer such that $m_{\alpha}(I_a) \neq 0$, 
$M_\alpha$ consists of $2^{n-\alpha-1}$ $0$s followed by $a-2^{n-\alpha-1} = m_\alpha(I_a)$ $1$s.   
Now for all $M_{\beta}$ with $\beta > \alpha$, $m_\alpha \leq m_\beta$ by the ordering lemma.
\end{proof}

\begin{defn}[Proper Labeling]
A proper m-table labeling in $Q_n$ is a labeling that assigns each label from $\{0, 1, \ldots, n-1\}$ to exactly one entry of the $2$ by $n$ table, with one label per column.
\end{defn}

\noindent There is a natural bijection between $Aut(Q_n)$ and the set of proper $m$-table labelings. In particular, if $g = x \oplus \sigma \in Aut(Q_n)$, where $x = x_0 x_1 \ldots x_{n-1}$, then the corresponding proper m-table labeling is the labeling that assigns label $i \in \{0, 1, \ldots, n-1\}$ to the row $x_i$ and the column $\sigma^{-1}(i)$. If $S \subseteq Q_n$, and $a_i$ is the entry in the m-table of S labeled $i$, then $m(g(S)) = (a_0, a_1, \ldots, a_{n-1})$. Thus, the action of $g$ on $m(I_a)$ can be visually represented by the m-table.

\begin{defn}
Let $S_1 \subseteq Q_n$, $g \in Aut(Q_n)$. Then the \textbf{g-labeled m-table} of $S_1$ will denote the m-table of $S_1$ with labeling given by $g_1$. If the entries $(i_1, j_1)$ in the labeled m-table of $S_1$ under $g$ and $(i_2, j_2)$ in the labeled m-table of $S_2$ under $g'$ are given the same label, then $(i_1, j_1)$ is said to be \textbf{paired} with $(i_2, j_2)$.
\end{defn}

\begin{exmp} Let $x = 101 \in Q_3$, and $\sigma = (012) \in S_3$ in cycle notation. Then the m-table of $I_6$ and the corresponding labeling is

$$\left( \begin{array}{ccc} 2^{[1]} & 2 & 3 \\
4 & 4^{[2]} & 3^{[0]} \end{array} \right).$$
The information in $m(g_1(I_a)) + m(g_2(I_b)) = m(I_{a+b})$ can also be encoded in the $m$-table. For example,

$$\left( \begin{array}{ccccc} 0^{[1]} & 1^{[2]} & 4 & 4^{[3]} & 4^{[4]} \\
9 & 8 & 5^{[0]} & 5 & 5 \end{array} \right) \text{  and  } \left( \begin{array}{ccccc} 0^{[0]} & 4 & 4 & 6^{[3]} & 6^{[4]} \\
12 & 8^{[1]} & 8^{[2]} & 6 & 6 \end{array} \right)$$

\noindent indicates that $m(g_1(I_9)) + m(g_2(I_{12})) = m(I_{21})$ for the $g_1$ and $g_2$ corresponding to the labeling on the $m$-table, since $m(I_{21}) = (5, 8, 9, 10, 10)$ and that the labeled pairs add up correctly. \\ \end{exmp}

\noindent We end this section by an observation that will be often used without mentioning. Let $I_a \subseteq Q_n$, $g \in Aut(Q_n)$. If the first $k$ entries in $m(g(I_a))$ are $0$, then there exists $g' \in Aut(Q_n)$ such that $g(I_a) = g'(I_a)$ and $g'$ restricts to an automorphism of $Q_{n-k}$. This $g'$ can be obtained by acting as identity on the first $k$ coordinates and copying the action of $g$ on the last $n-k$ coordinates. Therefore, when $g$ does not affect the first $k$ entries, we may without loss of generality identify $g$ as an element of $Aut(Q_{n-k})$.\\ \\

\section{Proof of Main Theorem}
\noindent In this section, we'll state and prove our main result, which gives a criterion to decide whether a pair $(a, b)$ fits in $Q_n$. This result allows us to fully classify all unfit pairs.

\begin{thm}
\label{main theorem}
Let $I_a, I_b, I_{a+b} \subseteq Q_n$, and $g_1, g_2 \in Aut(Q_n)$. Then $m(g_1(I_a)) + m(g_2(I_b)) = m(I_{a+b})$ if and only if $g_1(I_a) \cup g_2(I_b) = I_{a+b}$.
\end{thm}

\begin{defn}
\label{hyperface}
Let $F_{j, k} \subseteq Q_n$ denote the hyperface
$$\{x_0 x_1 \ldots x_{n-1} \in Q_n: x_j = k\}.$$
\end{defn}

\begin{lem}
\label{hyperface rep} Let $S \subseteq Q_n$. Then $m_j(S) = |S \cap F_{j, 1}|$, and $|S| - m_j(S) = |S \cap F_{j, 0}|$. In particular, if $I_a, I_b, I_{a+b} \subseteq Q_n$, and $g_1, g_2 \in Aut(Q_n)$, then $m(g_1(I_a)) + m(g_2(I_b)) = m(I_{a+b})$ if and only if
$$|g_1(I_a) \cap F_{j,k}| + |g_2(I_b) \cap F_{j,k}| = |I_{a+b} \cap F_{j,k}|$$
for each hyperface $F_{j,k}$.
\end{lem}

\begin{proof} The first statement follows from the definition of the hyperface $F_{j, 1}$. $F_{j, 1}$ consists of the elements of $Q_n$ which have a $1$ in position $j$. Therefore, $S \cap F_{j, 1}$ are the elements in $S$ with a $1$ in position $j$. $F_{j, 1}$ and $F_{j, 0}$ partition $Q_n$ so both equalities follow. The second statement is then immediate.
\end{proof}

\noindent Next, we have a lemma allowing us to restrict our attention to the hyperface $F_{0, 0}$, which will help to simplify our proof.

\begin{lem} 
\label{Reduce to hyperface}
Let $I_a, F_{j,k} \subseteq Q_n$, $g_1 \in Aut(Q_n)$, and $a' = |g_1(I_a) \cap F_{j,k}|$. There is a map $h \in Aut(Q_n)$ such that $h(F_{0, 0}) = g_1^{-1}(F_{j,k})$ and $h(I_{a'}) = I_a \cap g_1^{-1}(F_{j,k})$.
\end{lem}
\begin{proof} First, note that the image of a hyperface under automorphism is also a hyperface. Therefore, let $F_{j', k'} := g_1^{-1}(F_{j, k})$ and define $x$ to be the largest natural number in the set $S := I_a \cap F_{j', k'}$. All $a' - 1$ other elements of $S$ are less than $x$ in the binary order by definition, but furthermore these must be all of the elements in  $F_{j', k'}$ less than $x$ since $S$ is a subset of an initial segment. Therefore, $S \cong I_{a'}$ and since $a' \leq 2^{n - 1}$, $I_{a'} \subset Q^{n - 1}$. Thus, the automorphism $h$ given by Proposition \ref{transitiveaction} must send $F_{0, 0} = Q^{n - 1}$ to $g_1^{-1}(F_{j, k})$ as desired.
\end{proof}

\begin{prop}
\label{main lemma}
Let $I_a, I_b, I_{a+b} \subseteq Q_n$, and $g_1, g_2 \in Aut(Q_n)$. Suppose $2^{n-1} < a + b \leq 2^n$. If $m(g_1(I_a)) + m(g_2(I_b)) = m(I_{a+b})$, then there is a hyperface $F_{j,k}$ such that 
$$m(g_1(I_a) \cap F_{j,k}) + m(g_2(I_b) \cap F_{j,k}) = m(I_{a+b} \cap F_{j,k}).$$
\end{prop}

\noindent We'll first show that Proposition \ref{main lemma} implies Theorem \ref{main theorem}. If $g_1(I_a) \cup g_2(I_b) = I_{a+b}$, then $g_1(I_a)$ and $g_2(I_b)$ are disjoint, and thus, $m(g_1(I_a)) + m(g_2(I_b)) = m(I_{a+b})$. \\

\noindent Suppose $m(g_1(I_a)) + m(g_2(I_b)) = m(I_{a+b})$, $I_a, I_b, I_{a+b} \subseteq Q_n$, and $g_1, g_2 \in Aut(Q_n)$. We'll use induction on $n$. When $n=1$, the statement is clear. Suppose that the statement is true for $I_a, I_b, I_{a+b} \subseteq Q_{n-1}$. If $a + b \leq 2^{n-1}$, and $I_a, I_b, I_{a+b} \subseteq Q_n$, then since every $g_1, g_2 \in Aut(Q_n)$ that forces $m(g_1(I_a)) + m(g_2(I_b)) = m(I_{a+b})$ may only change the values of the last $n-1$ coordinates of $m(I_a)$ and $m(I_b)$, we can treat $g_1, g_2$ as automorphisms of $Q_{n-1}$ and treat $I_a, I_b, I_{a+b}$ as subsets of $Q_{n-1}$. By inductive hypothesis, $g_1(I_a) \cup g_2(I_b) = I_{a+b}$. \\

\noindent Now suppose $2^{n-1} < a+b \leq 2^n$. By Proposition \ref{main lemma}, the conditions imply that there is a hyperface $F_{j,k}$ such that $$m(g_1(I_a) \cap F_{j,k}) + m(g_2(I_b) \cap F_{j,k}) = m(I_{a+b} \cap F_{j,k}).$$

\noindent Let $a' = |g_1(I_a) \cap F_{j,k}|$ and $b' = |g_2(I_b) \cap F_{j,k}|$. By Lemma \ref{hyperface rep}, $$|I_{a+b} \cap F_{j,k}| = |g_1(I_a) \cap F_{j,k}| + |g_2(I_b) \cap F_{j,k}| = a' + b'.$$

\noindent By Lemma \ref{Reduce to hyperface}, $g_1(I_a) \cap F_{j,k}$, $g_2(I_b) \cap F_{j,k}$, and $I_{a+b} \cap F_{j,k}$ are isomorphic to $I_{a'}$, $I_{b'}$, and $I_{a'+b'}$, respectively. In other words, there are $h_1, h_2, h_3 \in Aut(Q_n)$ such that $$m(h_1(I_{a'})) + m(h_2(I_{b'})) = m(h_3(I_{a'+b'})).$$

\noindent By Lemma \ref{left multiplication}, $$m(h_3^{-1} h_1(I_{a'})) + m(h_3^{-1} h_2(I_{b'})) = m(I_{a'+b'}).$$

\noindent Note that $a' + b' \leq |F_{j,k}| = 2^{n-1}$. Hence, regarding everything to be in $Q_{n-1}$, the inductive hypothesis implies that $$h_3^{-1} h_1(I_{a'}) \cup h_3^{-1} h_2(I_{b'}) = I_{a'+b'}.$$

\noindent So $h_1(I_{a'}) \cup h_2(I_{b'}) = h_3(I_{a'+b'})$, which gives $$(g_1(I_a) \cap F_{j,k}) \cup (g_2(I_b) \cap F_{j,k}) = I_{a+b} \cap F_{j,k}.$$

\noindent Now $m(g_1(I_a) \cap F_{j,1-k}) + m(g_2(I_b) \cap F_{j,1-k}) = m(g_1(I_a)) - m(g_1(I_a) \cap F_{j,k}) + m(g_2(I_b)) - m(g_2(I_b) \cap F_{j,k}) = m(I_{a+b}) - m(I_{a+b} \cap F_{j,k}) = m(I_{a+b} \cap F_{j, 1-k})$. By the same argument, $$(g_1(I_a) \cap F_{j,1-k}) \cup (g_2(I_b) \cap F_{j,1-k}) = I_{a+b} \cap F_{j,1-k}$$

\noindent Taking union of the above equations give $g_1(I_a) \cup g_2(I_b) = I_{a+b}$. \\

\noindent We'll use the rest of the section to prove Proposition \ref{main lemma}. Unless otherwise noted, $I_a, I_b, I_{a+b} \subseteq Q_n$, $g_1, g_2 \in Aut(Q_n)$, and $2^{n-1} < a+b \leq 2^n$. Furthermore, from now on let $a' = m_0(g_1(I_a))$ and $b' = m_0(g_2(I_b))$. Since $a+b > 2^{n-1}$, $F_{0,0} \subseteq I_{a+b}$. \\

\noindent \textbf{Case 1:} $\max \{a, b\} < 2^{n-1}$, and one of $a', b' = 0$. \\

\noindent Without loss of generality, assume $a' = 0$. Then $g_1(I_a) \cap F_{0, 1} = \emptyset$, and $g_1(I_a) \cap F_{0,0} = g_1(I_a)$. We can partition $I_{a+b}$ into $g_1(I_a)$, $F_{0,0} - g_1(I_a)$, and $I_{a+b} \cap F_{0,1}$. All of the above disjoint subsets of $I_{a+b}$ are isomorphic to initial segments, so we write $F_{0,0} - g_1(I_a) = g_1'(I_{2^{n-1}-a})$ and $I_{a+b} \cap F_{0,1} = g_3(I_{a+b-2^{n-1}})$. Since the subsets are disjoint, we have $$m(I_{a+b}) = m(g_1(I_a)) + m(g_1'(I_{2^{n-1}-a})) + m(g_3(I_{a+b-2^{n-1}}))$$

\noindent By assumption, $m(I_{a+b}) = m(g_1(I_a)) + m(g_2(I_b))$. The two equality gives
$$m(g_2(I_b)) = m(g_1'(I_{2^{N-1}-a})) + m(g_3(I_{a+b-2^{N-1}})).$$

\noindent So $$m(I_b) = m(g_2^{-1}g_1'(I_{2^{N-1}-a})) + m(g_2^{-1}g_3(I_{a+b-2^{N-1}})).$$

\noindent Note that $g_2^{-1}g_1'(I_{2^{N-1}-a})$ and $g_2^{-1}g_3(I_{a+b-2^{N-1}})$ are disjoint, and the union of the two sets contains exactly $b$ elements. Thus, by Lemma \ref{uniqueness prop}, $I_b = g_2^{-1}g_1'(I_{2^{N-1}-a}) \cup g_2^{-1}g_3(I_{a+b-2^{N-1}})$, so
$$g_2(I_b) = g_1'(I_{2^{N-1}-a}) \cup g_3(I_{a+b-2^{N-1}}).$$

\noindent In particular, $g_2(I_b) \cap F_{0,1} = g_3(I_{a+b-2^{N-1}}) = I_{a+b} \cap F_{0,1}$, so $$m(g_1(I_a) \cap F_{0,1}) + m(g_2(I_b) \cap F_{0,1}) = m(g_2(I_b) \cap F_{0,1}) = m(I_{a+b} \cap F_{0,1}),$$ which proves Proposition \ref{main lemma} for this case.
\\ \\

\noindent \textbf{Case 2:} $\max\{a,b\} < 2^{n-1}$, and $a', b' > 0$. \\

\noindent We first state some computational results that will be used in the proof of Case 2. 

\begin{lem}
\label{computational results}
Let $I_a \subseteq Q_n$.
\begin{enumerate}
\item For each $j$, $m_j(I_a) \leq m_{n-1}(I_a) = \lfloor \frac{a}{2} \rfloor \leq \frac{a}{2}$.
\item If $a > 2^{n-1}$, $m_j(I_a) \geq 2^{n-2}$ for $j \geq 1$.
\item If $a > 2^{n-1} + 2^{n-3}$, then $m_2(I_a) > 2^{n-2}$.
\item If $a < 2^{n-2} + 2^{n-3}$, then $m_i(I_a) \leq a - m_i(I_a) < 2^{n-2}$ for $i \geq 2$.
\end{enumerate}
\end{lem}
\begin{proof}

To (1), it follows from Lemma \ref{orderinglemma} and the fact that $m_{n - 1}(I_a)$ counts the number of $1$s in the rightmost place from $0$ to $a - 1$, i.e. the number of natural numbers congruent to $1 (mod 2)$ in this range. To (2), recall that for each $j \geq 1$ $m_j(I_a)$ counts the number of $1$s in the $j$th position. Thus, for each block of $2^{n - j}$ elements in $I_a$, exactly $2^{n - j - 1}$ of them will have a $1$ in the $j$th position, from which the statement follows. To (3), $a > 2^{n - 1} + 2^{n - 3}$, the numbers in the range $2^{n - 1}$ to $2^{n - 1} + 2^{n - 3} - 1$ all have no $1$ in the second position. Starting at $2^{n - 1} + 2^{n - 3}$, there will be a $1$ in the second position, from which the statement follows. To (4), the first inequality follows from (1) above. The second follows from the fact that $m_i(I_a) \geq 2^{n - 3}$ for $a \geq 2^{n - 2}$, $i \geq 2$ by a similar argument to statement (2) above.
\end{proof}

\begin{lem} 
\label{zero matching}
Let $I_a, I_b, I_{a+b} \subseteq Q_n$, and $g_1, g_2 \in Aut(Q_n)$. Let $2^{n-1} < a + b \leq 2^n$, $\max\{a,b\} < 2^{n-1}$, and $a', b' > 0$. Suppose $m(g_1(I_a)) + m(g_2(I_b)) = m(I_{a+b})$. Then
\begin{enumerate}
\item Let $\alpha, \beta$ be the smallest integer such that $m_{\alpha}(I_a), m_{\beta}(I_b) \neq 0$, respectively. Then $\alpha = \beta = 1$, and $a' = a - 2^{n-2}$, $b' = b - 2^{n-2}$. Thus, $a, b > 2^{n-2}$.
\item  There exist $g_1',g_2' \in Aut(Q_n)$ such that $m(g_1'(I_a)) = m(g_1(I_a))$, $m(g_2'(I_b)) = m(g_2(I_b))$, and the entry $(1,1)$ in the m-tables for $I_a$ (and respectively, $I_b$) is labeled $0$ by $g_1$ (and respectively, $g_2$).
\item Suppose $g_1$ and $g_2$ satisfies the conditions in (2). Then the labeled entry in column 0 of the m-table of $I_a$ is paired with the labeled entry in column 0 of the m-table of $I_b$. Furthermore, if the label is $j$, then either $m_j(g_1(I_a)) = a$ and $m_j(g_2(I_b)) = 0$ or $m_j(g_1(I_a)) = 0$ and $m_j(g_2(I_b)) = b$.
\end{enumerate}
\end{lem}

\begin{proof} To prove (1), first note that by Lemma \ref{m-table min}, the smallest non-zero number in the m-tables of $I_a$ and $I_b$ are $a - 2^{n-1-\alpha}$ and $b - 2^{n-1-\beta}$, respectively. Hence, $$a'+b' \geq a + b - 2^{n-1-\alpha} - 2^{n-1-\beta}$$

\noindent But $a' + b' = m_0(I_{a+b}) = a + b - 2^{n-1}$, so $$2^{n-1-\alpha} + 2^{n-1-\beta} \geq 2^{n-1}$$

\noindent Since $\max\{a,b\} < 2^{n-1}$, $\alpha, \beta > 0$. Thus, the inequality holds if and only if $\alpha = \beta = 1$. This proves that $a' = a - 2^{n-2}$, $b' = b - 2^{n-2}$ and finishes the proof of (1). \\

\noindent For (2), it suffices to show that $g_1$ can be replaced by a $g_1'$ with the desired property. Suppose the label $k \neq 0$ is assigned to the entry $(i_1, 1)$ and that the label $0$ is assigned to the entry $(i_2, j_2) \neq (1, 1)$ in the labeled m-table of $I_a$ under $g_1$. Now we relabel the m-table of $I_a$ such that the label $0$ is assigned to (1, 1) and the label $k$ is assigned to $(i_1, j_2)$. This new label is also proper, so there is an corresponding automorphism $g_1'$. For each $k' \not\in \{0, k\}$, $m_{k'}(g_1'(I_a)) = m_{k'}(g_1(I_a))$ since the position of the label $k'$ does not change. Moreover, $m_0(g_1'(I_a)) = m_1(I_a)$, and since $2^{n-2} < a < 2^{n-1}$, we have $$m_0(g_1'(I_a)) = m_1(I_a) = a - 2^{n-2} = m_0(g_1(I_a))$$

\noindent by part (1). Lastly, in our new labeling, $m_k(g_1'(I_a))$ is the entry $(i_1, j_2)$, and $m_k(g_1(I_a))$ is the entry $(i_1, 1)$ in the m-table of $I_a$. These entries are equal, since $m_{j_2}(I_a) = m_0(g_1(I_a)) = m_1(I_a)$, which implies that the column $1$ and column $j_2$ are identical. Hence, $m_k(g_1(I_a)) = m_k(g_1'(I_a))$, and so $m(g_1(I_a)) = m(g_1'(I_a))$. Using the same argument on $I_b$ and $g_2$ gives the automorphism $g_2'$. \\

\noindent For (3), suppose that the label in column $0$ of the $g_1$-labeled (respectively, $g_2$-labeled) m-table for $I_a$ (respectively, $I_b)$ is $j_1$ (respectively, $j_2$). We first note  that $m_{j_1}(g_2(I_b)) = 0$ implies that $j_1 = j_2$. This is because only column $0$ in the m-table of $I_a$ contains the value $0$, so $g_2$ can only assign label $j_1$ to column $0$, which by definition implies that $j_1 = j_2$. In the same way, if $m_{j_2}(g_1(I_a)) = 0$, then the conclusion holds. \\

\noindent Moreover, $j_1 = j_2$ implies that exactly one of $m_{j_1}(g_1(I_a)) = a$ or $m_{j_2}(g_2(I_b)) = b$ is true. To see this, observe that if neither is true, then $m_{j_1}(g_1(I_a)) + m_{j_1}(g_2(I_b)) = m_{j_1}(I_{a+b}) = 0$. We know that $j_1 \neq 0$ by part (2), but for such $j_1$, the condition that $a+b > 2^{N-1}$ guarantees that $m_{j_1}(I_{a+b}) > 0$, which is a contradiction. If both are true, then $m_{j_1}(I_{a+b}) = a + b$ which is strictly greater than $(a+b)/2$. But by Lemma \ref{computational results}(1), $m_{j_1}(I_{a+b}) \leq m_{n-1}(I_{a+b}) \leq (a+b)/2$, which is a contradiction. Thus, it suffices to prove that $m_{j_1}(g_2(I_b)) = 0.$\\

\noindent First, consider the case where at least one of $m_{j_1}(g_1(I_a))$ and $m_{j_2}(g_2(I_b))$ is non-zero. Without loss of generality, suppose the former is non-zero, so $m_{j_1}(g_1(I_a)) = a$. Suppose for the sake of contradiction that $m_{j_1}(g_2(I_b)) \neq 0$. Then, 
$$m_{j_1}(I_{a+b}) \geq a + b' = a + b - 2^{n-2}$$
But $m_{j_1}(I_{a+b}) \leq (a+b)/2$, so
$$a + b - 2^{n-2} \leq \frac{a+b}{2}$$
which implies that $a+b \leq 2^{n-1}$, contradicting our hypothesis. \\

\noindent We're then left to prove the case when both $m_{j_1}(g_1(I_a)) = 0$ and $m_{j_2}(g_2(I_b)) = 0$. Since the label $0$ is already used, $j_1, j_2 \geq 1$. By Lemma \ref{computational results}(2), $m_j(I_{a+b}) \geq 2^{n-2}$ for each $j \geq 1$. It follows that 
$$m_{j_1}(g_2(I_b)) \geq 2^{n-2} \text{ and } m_{j_2}(g_1(I_a)) \geq 2^{n-2}.$$
If $m_{j_1}(g_2(I_b)) > 2^{n-2}$, then $m_{j_1}(g_2(I_b)) = b$, which by definition implies that $j_1 = j_2$. Similarly, $m_{j_2}(g_1(I_a)) > 2^{n-2}$ also implies that $j_1 = j_2$. Hence, we'll assume for the rest of the proof that 
$$m_{j_1}(g_2(I_b)) \leq 2^{n-2} \text{ and } m_{j_2}(g_1(I_a)) \leq 2^{n-2}.$$
Combining the inequalities, we see that 
$$m_{j_1}(g_2(I_b)) = m_{j_2}(g_1(I_a)) = m_{j_1}(I_{a+b}) = m_{j_2}(I_{a+b}) = 2^{n-2}.$$
By Lemma \ref{computational results}(3), when $a + b > 2^{n-1} + 2^{n-3}$, $m_2(I_{a+b}) > 2^{n-2}$. In this case, only $m_0(I_{a+b})$ and $m_1(I_{a+b})$ could possibly be $2^{n-2}$. But since $j_1$ and $j_2$ are distinct and non-zero, this could never happen. Hence, $a + b \leq 2^{n-1} + 2^{n-3} < 2^{n-1} + 2^{n-2}$. \\

\noindent Suppose $g_1$ assigns label $j_2$ to column $i_1$, and $g_2$ assigns label $j_1$ to column $i_2$. Since both column $0$ and column $1$ already have other labels in our current assumptions, $i_1, i_2 \geq 2$. This means that $m_{i_1}(I_a)$ or $a - m_{i_1}(I_a)$ is equal to $2^{n-2}$, and $m_{i_2}(I_b)$ or $b - m_{i_2}(I_b)$ is equal to $2^{n-2}$. By Lemma \ref{computational results}(4), this takes place only when $a, b \geq 2^{n-2} + 2^{n-3}$. But then $a+b \geq 2^{n-1} + 2^{n-2}$, which contradicts the above bound. This concludes the proof.
\end{proof}

\noindent Now we prove the proposition for case 2. By Lemma \ref{zero matching}(2), we can assume without loss of generality that the entries $(1,1)$ in the m-table of $I_a$ (and $I_b$, respectively) is labeled $0$ by $g_1$ (and $g_2$, respectively). This means that $g_1^{-1}(F_{0, 1}) = g_2^{-1}(F_{0,1}) = F_{1,1}$. Equivalently, $g_1^{-1}(F_{0,0}) = g_2^{-1}(F_{0,0}) = F_{1,0}$. We'll show that $$m(g_1(I_a \cap F_{1,0})) + m(g_2(I_b \cap F_{1,0})) = m(I_{a+b} \cap F_{0,0}),$$
\noindent which would imply that $m(g_1(I_a) \cap F_{0,0}) + m(g_2(I_b) \cap F_{0,0}) = m(I_{a+b} \cap F_{0,0})$. \\

\noindent By Lemma \ref{zero matching}(3), if an entry in the column $0$ of $I_a$ is labeled $j$ by $g_1$, then an entry in the column $0$ of $I_b$ is also labeled $j$ by $g_2$. Furthermore, exactly one of the two entries lie in row $0$. Without loss of generality, assume that the entry in $(0, 0)$ in $I_a$ and the entry $(1, 0)$ in $I_b$ are labeled $j$ by $g_1, g_2$, respectively. Then the labeled m-tables become:
$$\left( \begin{array}{ccccc} 0^{[j]} & a - 2^{n-2^{[0]}} & * & \ldots & * \\
a & 2^{n-2} & * & \ldots & * \end{array} \right) \text{ and } \left( \begin{array}{ccccc} 0 & b - 2^{n-2^{[0]}} & * & \ldots & * \\
b^{[j]} & 2^{n-2} & * & \ldots & * \end{array} \right)$$

Note that $I_a \cap F_{1,0}$ and $I_b \cap F_{1,0}$ are precisely the initial segment of length $2^{n-2}$. Hence, the labeled m-table of $I_a \cap F_{1,0}$ by $g_1$ and $I_b \cap F_{1,0}$ by $g_2$ are:
$$\left( \begin{array}{ccccc} 0^{[j]} & 0^{[0]} & 2^{n-3} & \ldots & 2^{n-3} \\
2^{n-2} & 2^{n-2} & 2^{n-3} & \ldots & 2^{n-3} \end{array} \right) \text{ and } \left( \begin{array}{ccccc} 0 & 0^{[0]} & 2^{n-3} & \ldots & 2^{n-3} \\
2^{n-2^{[j]}} & 2^{n-2} & 2^{n-3} & \ldots & 2^{n-3} \end{array} \right)$$

\noindent Thus, $m(g_1(I_a) \cap F_{0,0}) + m(g_2(I_b) \cap F_{0,0}) = (0, 2^{n-2}, 2^{n-2}, \ldots, 2^{n-2}) = m(I_{a+b} \cap F_{0,0})$. This finishes the proof of case 2. \\ \\

\noindent \textbf{Case 3: } $\max \{a,b\} \geq 2^{n-1}$. \\

\noindent Without loss of generality, let $a \geq b$. Then $a \geq 2^{n-1}$, so if $F_{j,k} = g_1(F_{0,0})$, then 
$$|g_1(I_a) \cap F_{j,k}| = |g_1(I_a \cap F_{0,0})| = 2^{n-1}$$
By Lemma \ref{hyperface rep}, $m(g_1(I_a)) + m(g_2(I_b)) = m(I_{a+b})$ implies that 
$$|g_1(I_a) \cap F_{j,k}| + |g_2(I_b) \cap F_{j,k}| = |I_{a+b}\ \cap F_{j,k}|$$
But $|I_{a+b}\ \cap F_{j,k}| \leq |F_{j,k}| = 2^{n-1}$, which forces $|g_2(I_b) \cap F_{j,k}| = 0$ and $|I_{a+b}\ \cap F_{j,k}| = 2^{n-1}$. Hence, 
$$g_1(I_a) \cap F_{j,k} = F_{j,k} \text{, } g_2(I_b) \cap F_{j,k} = \emptyset \text{, and } I_{a+b} \cap F_{j,k} = F_{j,k}$$
which clearly implies the conclusion. This finishes the proof of Proposition \ref{main lemma}, and hence, the proof of Theorem \ref{main theorem}. \\

\section{Classification of Unfit Pairs}
\noindent The main theorem allows us to efficiently compute unfit pairs for small $n$. For example, the set of unfit pairs $(a, b)$, where $a, b \leq 2^6$, are shown in Figure \ref{counterexamples}, where the cell in the $a$th row and $b$th column is colored black if $(a, b)$ is an unfit pair. In this section and the next section, we'll fully classify the unfit pairs. Unless otherwise noted, every initial segment and other subsets of the hypercube mentioned in section 4 and 5 is contained in $Q_{n+1}$.

\begin{figure} \label{counterexamples}
\includegraphics[width=15cm]{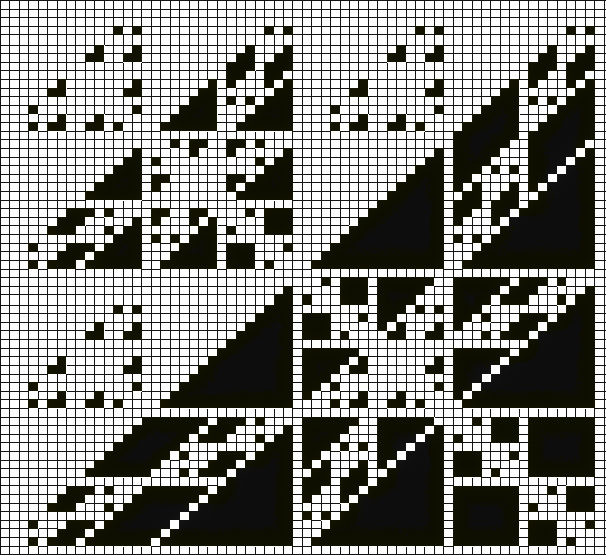}
\caption{Unfit pairs for $a, b \leq 2^6$}
\end{figure}

\subsection{Translation of the Upper Triangle}
\begin{prop}
\label{Translation of the Upper Triangle}
Suppose $a + b \leq 2^n$. Then $(a, b)$ fits if and only if $(a + 2^n, b)$ fits if and only if $(a, b + 2^n)$ fits.
\end{prop}

\begin{proof}
If $(a,b)$ fits, where $g_1(I_a) \cup g_2(I_b) = I_{a+b}$, then $g_1$ and $g_2$ affect only the subset $Q_n = F_{0,0}$ in $Q_{n+1}$. Then clearly, $g_1(I_{2^n+a}) \cup 10\cdots 0 \oplus g_2(I_b) = I_{2^n+a+b}$. \\

\noindent Suppose now $(a+2^n, b)$ fits, where $g_1(I_{a+2^n}) \cup g_2(I_b) = I_{2^n+a+b}$. Let $g_1(F_{0,0}) = F_{j,k}$. Then  $|g_1(I_{a+2^n}) \cap F_{j,1-k}| = a$, and $g_2(I_b) \subseteq F_{j,1-k}$. Since $g_1(I_{a+2^n}) \cap F_{j,1-k}, g_2(I_b), I_{2^n+a+b} \cap F_{j,1-k}$ are isomorphic to the initial segments of lengths $a, b, a+b$, respectively, $(a,b)$ fits.
\end{proof}

\subsection{Reflection of the Lower Triangle}
\begin{prop}
\label{Reflection of the Lower Triangle}
Suppose $a, b \leq 2^{n-1}$, and $2^{n-1} \leq a + b \leq 2^n$. Then $(a, b)$ fits if and only if $(2^n - a, 2^n - b)$ fits. Thus, the number of unfit pairs in $\{(a,b): a, b \leq 2^{n-1}, 2^{n-1} \leq a+b \leq 2^n\}$ is equal to the number of unfit pairs in $\{(a,b): a, b \geq 2^{n-1}, a+b \leq 2^n+2^{n-1}\}.$
\end{prop}

\begin{proof}
We first show that if there are $g_1, g_2 \in \Aut(Q_n)$ such that $g_1(I_a) \cup g_2(I_b) = I_{a+b}$, and $a, b \leq 2^{n-1}$, $2^{n-1} \leq a+b \leq 2^n$, then $(2^n-a, 2^n-b)$ fits. Let $a' = m_1(g_1(I_a))$ (which is the first non-zero entry in $m(I_{a+b})$ given the bounds of $a, b$), and $b' = m_1(g_2(I_b))$. \\

\noindent Suppose $a' = 0$ or $b' = 0$. Without loss of generality, suppose $a' = 0$. We'll need the following lemma. 

\begin{lem}
\label{Shifting m-table} 
As usual, assume that all relevant initial segments lie in $Q_{n+1}$. Let $k \in \Z$, $a \leq 2^k$. Then the last $k$ columns of the $m$-table of $I_{2^k-a}$ can be obtained by adding each corresponding entry of the $m$-table of $I_a$ by $2^{n-1} - a$.
\end{lem}

\begin{proof} First, since $m_j$ measures the number of $1$'s on a hyperface, $m_j(Q_k - I_a) = 2^{k-1} - m_j(I_a)$. Observe that $I_{2^k-a}$ can be obtained by turning all $0$'s to $1$'s and all $1$'s to $0$'s, which is simply applying the automorphism $x \rightarrow 1\cdots 1 \oplus x$. Thus, 
$$m_j(I_{2^k-a}) = m_j(1\cdots 1 \oplus (Q_k-I_a)) = (2^k-a) - (2^{k-1} - m_j(I_a)) = (2^{k-1}-a) + m_j(I_a).$$
This gives the conclusion for the first row of the m-table. The entry on the $j$th column and second row of the m-table is 
$$(2^k-a) - m_j(I_{2^k-a}) = 2^{k-1} - m_j(I_a) = (2^{k-1}-a) + (a-m_j(I_a)),$$
and $a - m_j(I_a)$ is the corresponding entry in the m-table for $I_a$. This concludes the proof.
\end{proof}

\noindent In particular, the above lemma gives that $m_j(I_{2^{n+1}-a-b}) = (2^n-a-b) + m_j(I_{a+b})$ for each $j$. \\

\noindent We next claim that for $j \geq 1$, $$m_j(g_1(I_{2^n-a})) + m_j(g_2(I_{2^n-b})) = m_j(I_{2^{n+1}-a-b}).$$ Since $I_{a+b} \subseteq Q_n$, $m_0(g_1(I_a)) = m_0(g_2(I_b)) = 0$. By the remark at the end of section 2, we can treat $g_1, g_2$ as automorphism of $Q_n$. In particular, we can without loss of generality give the label $0$ to the entry on the first row and first column. Then all non-zero labels are assigned to the last $n$ columns. By the preceding lemma, 
$$m_j(g_1(I_{2^n-a})) = (2^{n-1}-a) + m_j(g_1(I_a)),$$
$$m_j(g_2(I_{2^n-b})) = (2^{n-1}-b) + m_j(g_2(I_b)).$$
By the previous paragraph, $m_j(I_{2^{n+1}-a-b}) = (2^n-a-b) + m_j(I_{a+b})$ for each $j$, so the claim immediately follows. \\

\noindent Now we construct automorphisms $h_1, h_2 \in \Aut(Q_{n+1})$ that makes $I_{2^n-a}$ and $I_{2^n-b}$ fit by describing its labeling on $m$-table. First, give labels $2, 3, \ldots, n$ as in $g_1$ and $g_2$, which ensures that $m_j(g'_1(I_{2^n-a}))+m_j(g'_2(I_{2^n-b})) = m(I_{2^{n+1}-a-b})$. Then, consider the columns with label $0$ and $1$ in the original $g_1$-labeled and $g_2$-labeled m-table. Suppose such columns are column $0$ and $i$ in the $m$-table of $I_a$ and column $0$ and $j$ in the $m$-table of $I_b$. Then the $m$-tables are:

$$\left( \begin{array}{cc} 0 & m_i(I_a) \\
a & a-m_i(I_a) \end{array} \right) \text{ and } \left( \begin{array}{cc} 0 & m_j(I_b) \\ b & b-m_j(I_b) \end{array} \right).$$

\noindent By our assumption, $m_1(g_1(I_a)) = 0$, so the labeled entry in the second column is $0$, which must be on the first row. Thus, $m_i(I_a) = 0$. By the lemma above, the unlabeled parts of the $m$-table of $I_{2^n-a}$ and $I_{2^n-b}$ are:

$$\left( \begin{array}{cc} 0 & 2^{n-1}-a \\
2^n-a & 2^{n-1} \end{array} \right) \text{ and } \left( \begin{array}{cc} 0 & 2^{n-1}-b+m_j(I_b) \\
2^n-b & 2^{n-1}-m_j(I_b) \end{array} \right).$$

\noindent We know that $m_0(I_{2^{n+1}-a-b}) = 2^n-a-b$ and $m_1(I_{2^{n+1}-a-b}) = 2^{n-1}$. Thus, label $1$ must be given to the entries $2^{n-1}$ and $0$. \\

\noindent Observe that $m_i(I_a)=0$ is labeled $1$ in the original labeling. Since $m_i(I_{a+b}) = a+b-2^{n-1}$, the entry in the $m$-table of $I_b$ that is given the same label
must take the value $a+b-2^{n-1}$. The entry in the same column but different row is then $b-(a+b-2^{n-1}) = 2^{n-1}-a$. By the lemma, the corresponding entry in the $m$-table of $I_{2^n-b}$ is $(2^{n-1}-b)+(2^{n-1}-a) = 2^n-a-b$, which is precisely $m_0(I_{2^n-a-b})$. Thus, label $0$ must be given to the entries $0$ and $2^n-a-b$. This finishes the case where $a'=0$ or $b'=0$. \\

\noindent Now suppose $a', b' \neq 0$. Then, treating $I_a, I_b, I_{a+b}$ as subsets of $Q_n$ and $g_1, g_2$ as automorphism of $Q_n$, we can apply Lemma \ref{zero matching} and see that there is a column $j$ such that either $m_j(g_1(I_a)) = a$ and $m_j(g_2(I_b)) = 0$ or $m_j(g_1(I_a)) = 0$ and $m_j(g_2(I_b)) = b$. Without loss of generality, assume that the former is true. Then $g_1(I_a) \subseteq F_{j,1}$ and $g_2(I_b) \subseteq F_{j,0}$. \\

\noindent Observe that if $g_1(I_a) \subseteq F_{j,1}$, then $g(I_a) \cup F_{j, 0}$ is isomorphic to $I_{2^n + a}$. Fix $h \in \Aut(Q_{n+1})$ such that $g_1(I_a) \cup F_{j, 0} = h(I_{2^n+a})$. Since $g_2(I_b) \subseteq F_{j, 0}$, $F_{j, 0} - g_2(I_b)$ is an initial segment. Let $F_{j,0} - g_2(I_b) = h'(I_{2^n-b})$. Then $$h(I_{2^n+a}) = h'(I_{2^n-b}) \cup g_2(I_b) \cup g_1(I_a) = h'(I_{2^n-b}) \cup I_{a+b}$$
which means that $(2^n-b, a+b)$ fits. Since $(2^n-b) + (a+b) + (2^n-a) = 2^{n+1}$, $(2^n-a, 2^n-b)$ also fits. \\

\noindent To prove the converse, it is equivalent to prove the following statement: Suppose $2^{n-1} \leq a, b \leq 2^n$, and $2^n \leq a+b \leq 2^n + 2^{n-1}$. If $(a, b)$ fits, then $(2^n-a, 2^n-b)$ fits. \\

\noindent Suppose $a' = m_0(g_1(I_a))$ and $b' = m_0(g_2(I_b))$. Consider the case when $a' = 0$ or $b' = 0$, and without loss of generality, assume that $a' = 0$. Then $g_1(I_a) \subseteq F_{0,0}$, so $$g_2(I_b) = (F_{0,0} - g_1(I_a)) \cup (I_{a+b} \cap F_{0, 1})$$
Note that $F_{0,0} - g_1(I_a)$ is an initial segment of length $2^n-a$, and $I_{a+b} \cap F_{0,1}$ is an initial segment of length $a+b-2^n$, so $(2^n-a, a+b-2^n)$ fits. But $(2^n-a) + (a+b-2^n) + (2^n-b) = 2^n$, so $(2^n-a, 2^n-b)$ also fits. \\

\noindent Lastly, consider the case when $a', b' \neq 0$. Then, as above, we can assume without loss of generality that there is a column $j$ such that $m_j(g_1(I_a)) = a$ and $m_j(g_2(I_b)) = 0$. In the same way, there is $h \in \Aut(Q_{n+1})$ such that $g_1(I_a) \cup F_{j, 0} = h(I_{2^n+a})$. Now since $m_j(g_2(I_b)) = 0$, $g_2(I_b) \subseteq F_{j, 0}$, and $F_{j, 0} - g_2(I_b)$ is an initial segment. Let $F_{j,0} - g_2(I_b) = h'(I_{2^n-b})$. Then $$h(I_{2^n+a}) = h'(I_{2^n-b}) \cup g_2(I_b) \cup g_1(I_a) = h'(I_{2^n-b}) \cup I_{a+b}$$
which means that $(2^n-b, a+b)$ fits. Since $(2^n-b) + (a+b) + (2^n-a) = 2^{n+1}$, $(2^n-a, 2^n-b)$ also fits.
\end{proof}

\subsection{New Large Triangle}
\begin{prop}
\label{New Large Triangle}
The set 
$$\{(a, b): a < 2^n, b < 2^n - 2^{n-2}, a + b > 2^n + 2^{n-1}\}$$
consists of only unfit pairs. This counts $(4^{n-2} - 3\cdot 2^{n-2} + 2)/2$ unfit pairs.
\end{prop}

\begin{proof}
Set $c = 2^{n+1} - a - b$. By the bounds on $a$ and $b$, it is equivalent to prove that the set $$\{(b,c): 2^{n-1} < b < 2^n - 2^{n-2}, 2^{n-2} < c < 2^{n-1}, b+c > 2^n\}$$
consists only of unfit pairs. Suppose for the sake of contradiction that $m(g_1(I_b)) + m(g_2(I_c)) = m(I_{b+c})$. Since $2^{n-1} < b < 2^n - 2^{n-2}$ and $ 2^{n-2} < c < 2^{n-1}$, $m_i(I_b) \geq b - 2^{n-1}$ and $m_i(I_c) \geq c - 2^{n-2}$ if $m_i(I_b)$ and $m_i(I_c)$ are non-zero. \\

\noindent Since $c < 2^{n-1}$, $b+c < b+2^{n-1}$, so $b - 2^{n-1} > b + c -2^n$. Since $b < 2^n-2^{n-2}$, $b+c < 2^n - 2^{n-2} + c$, so $c - 2^{n-2} > b + c - 2^n$. Hence, if $m_i(g_1(I_b)) + m_i(g_2(I_c)) \neq 0$, then $m_0(g_1(I_b)) + m_0(g_2(I_c)) \geq \min \{b - 2^{n-1}, c - 2^{n-2}\} > b+c-2^n = m_0(I_{b+c})$. So $m_0(g_1(I_c)) = m_0(g_2(I_c)) = m_0(I_{b+c}) = 0$. But since $b + c > 2^n$, $m_0(I_{b+c}) \neq 0$, contradiction.
\end{proof}

\subsection{Remaining Cases}
\noindent We need to count the collection of squares at the bottom right of Figure \ref{counterexamples} that constitute the remaining unfit pairs. First, we prove some basic lemmas about the m-vectors. As usual, we assume that every subset of hypercubes lives in $Q_{n+1}$ so that the m-vector consists of $(m_0, m_1, \ldots, m_n)$.

\begin{lem}
\label{m of 2a and 2a+1}
\begin{enumerate}
\item Let $a \leq 2^n$. Then $$m_j(I_{2a}) = \left\{ \begin{array}{lcl}
2m_{j+1}(I_a) & \mbox{if}
& j \leq n - 1 \\ a & \mbox{if} & j = n
\end{array}\right.$$
\item Let $a < 2^n$. Then $$m_j(I_{2a+1}) = \left\{ \begin{array}{lcl}
m_{j+1}(I_a) + m_{j+1}(I_{a+1}) & \mbox{if}
& j \leq n - 1 \\ a & \mbox{if} & j = n
\end{array}\right.$$
\end{enumerate}
\end{lem}

\begin{proof} 
For (1), let $E \subseteq I_{2a}$ consist of all the even vertices (i.e., vertices that represent an even integer), and $O \subseteq I_{2a}$ consist of all the odd vertices. The claim follows by noting that for each $j < n$, $m_j(E) = m_j(O) = m_{j+1}(I_a)$, and when $j=n$, $m_j(I_{2a} = a$. For (2), we again write $E, O \subseteq I_{2a+1}$ as the collection of even and odd vertices, respectively. Then we have $m_j(E) = m_{j+1}(I_{a+1})$ and $m_j(O) = m_{j+1}(I_a)$ for $j<n$. For $j=n$, $m_j(I_{2a+1}) = \lfloor (2a+1)/2 \rfloor = a$.
\end{proof}

\noindent A quick examination of the $m$-tables gives the following corollary:

\begin{cor}
\label{Double a and b}
$(a, a), (a, a+1), (a, a-1)$ fits. Furthermore, if $(a, b)$ fits, then $(2a, 2b)$ fits.
\end{cor}

\begin{defn}
Let $a \in \N$. Define $v_2(a)$ as the largest $j$ such that $2^j$ divides $a$.
\end{defn}

\begin{prop}
Let $n \geq 2$, and $a, b \leq 2^n$. If $$|a - b| \leq 2^{\max \{v_2(a), v_2(b)\}},$$
then $(a, b)$ fits.
\end{prop}

\begin{proof}
Without loss of generality, assume that $a \geq b$. By the previous corollary, the case $|a-b| \leq 1$ is proved to be fit. Hence, we can assume that $a - b \geq 2$. \\

\noindent \textbf{Case 1:} $a - b \leq 2^{v_2(a)}$. \\

\noindent We'll prove that in this case, $$m_j(I_{a+b}) = \left\{ \begin{array}{lcl}
m_{j}(I_a) + m_{j}(I_b) & \mbox{if}
& j \geq n - v_2(a) \\ b & \mbox{if} & j = n - v_2(a) - 1 \\ m_{j+1}(I_a) + m_{j+1}(I_b) & \mbox{if} & j \leq n - v_2(a) - 2
\end{array}\right.$$

\noindent The case when $n=2$ can be readily verified. Suppose for induction that the above formula is true for $a, b \leq 2^{n-1}$. Now let $a, b \leq 2^n$. If $v_2(a) = 0$, then $a - b \leq 1$, but we've assumed that $a - b \geq 2$. Thus, $v_2(a) \geq 1$. If $v_2(b) \geq 1$ as well, then $a/2$ and $b/2$ are integers at most $2^{n-1}$, and $$\frac{a}{2} - \frac{b}{2} = \frac{1}{2}(a-b) \leq \frac{1}{2} 2^{v_2(a)} = 2^{v_2(a/2)}$$

\noindent By inductive hypothesis, $(a/2, b/2)$ satiesfies the above equation. Combining this with Lemma \ref{m of 2a and 2a+1}(1) (multiply everything by $2$) shows that the above equation also holds for $(a,b)$. Thus, we assume that $v_2(b) = 0$, i.e., $b$ is odd.\\

\noindent Let $a' = a/2$, $b' = (b-1)/2$. Since $a - 2^{v_2(a)}$ is even and $b$ is odd, $a - b \leq 2^{v_2(a)} \Longleftrightarrow b \geq a - 2^{v_2(a)}$ implies that $b-1 \geq a - 2^{v_2(a)}$. Thus, $$b' \geq a' - 2^{v_2(a) - 1} = a' - 2^{v_2(a')}$$

\noindent Furthermore, since $a-b \geq 2$, $b + 1 < a$, which means that $a' > b' + 1$. Moreover, $$b' + 1 = \frac{b}{2} + \frac{1}{2} \geq \frac{a}{2} - 2^{v_2(a)-1} + \frac{1}{2} > a' - 2^{v_2(a')}$$

\noindent By the inductive hypothesis and the inequalities above, and by the fact that $v_2(a') = v_2(a)-1$, we see that $$m_j(I_{a'+b'}) = \left\{ \begin{array}{lcl}
m_{j}(I_{a'}) + m_{j}(I_{b'}) & \mbox{if}
& j \geq n - v_2(a) + 1 \\ b' & \mbox{if} & j = n - v_2(a) \\ m_{j+1}(I_{a'}) + m_{j+1}(I_{b'}) & \mbox{if} & j \leq n - v_2(a) - 1
\end{array}\right.$$
$$m_j(I_{a'+b'+1}) = \left\{ \begin{array}{lcl}
m_{j}(I_{a'}) + m_{j}(I_{b'+1}) & \mbox{if}
& j \geq n - v_2(a) + 1 \\ b'+1 & \mbox{if} & j = n - v_2(a) \\ m_{j+1}(I_{a'}) + m_{j+1}(I_{b'+1}) & \mbox{if} & j \leq n - v_2(a) - 1
\end{array}\right.$$

\noindent By Lemma \ref{m of 2a and 2a+1}, we also have $$m_j(I_{a+b}) = m_j(I_{2(a'+b')+1}) = \left\{ \begin{array}{lcl}
m_{j+1}(I_{a'+b'}) + m_{j+1}(I_{a'+b'+1}) & \mbox{if}
& j \leq n - 1 \\ a'+b' & \mbox{if} & j = n
\end{array}\right.$$

\noindent Thus, when $j = n$, $m_j(I_{a+b}) = a' + b' = m_j(I_{a}) + m_j(I_b)$. When $n - v_2(a) \leq j < n$ (equivalently, $n-v_2(a)+1 \leq j+1 \leq n$), $$m_j(I_{a+b}) = m_{j+1}(I_{a'+b'}) + m_{j+1}(I_{a'+b'+1}) = m_{j+1}(I_{a'}) + m_{j+1}(I_{b'}) + m_{j+1}(I_{a'}) + m_{j+1}(I_{b'+1}) $$
$$= m_j(I_a) + m_j(I_b)$$

\noindent When $j = n - v_2(a) - 1$ (equivalently, $j+1 = n-v_2(a)$, $$m_j(I_{a+b}) = m_{j+1}(I_{a'+b'}) + m_{j+1}(I_{a'+b'+1}) = b' + (b' + 1) = 2b' + 1 = b$$

\noindent When $j \leq n - v_2(a) - 2$ (equivalently, $j+1 \leq n - v_2(a) - 1$), $$m_j(I_{a+b}) = m_{j+1}(I_{a'+b'}) + m_{j+1}(I_{a'+b'+1}) = m_{j+2}(I_{a'}) + m_{j+2}(I_{b'}) + m_{j+2}(I_{a'}) + m_{j+2}(I_{b'+1}) $$
$$= m_{j+1}(I_a) + m_{j+1}(I_b)$$

\noindent This completes the inductive step and proves the result for Case 1. \\

\noindent \textbf{Case 2:} $a - b \leq 2^{v_2(b)}$. \\

\noindent We'll prove that in this case, $$m_j(I_{a+b}) = \left\{ \begin{array}{lcl}
m_{j}(I_a) + m_{j}(I_b) & \mbox{if}
& j \geq n - v_2(b) \\ b & \mbox{if} & j = n - v_2(b) - 1 \\ m_{j+1}(I_a) + m_{j+1}(I_b) & \mbox{if} & j \leq n - v_2(b) - 2
\end{array}\right.$$

\noindent The case when $n=2$ can be readily verified. Suppose for induction that the above formula is true for $a, b \leq 2^{n-1}$. Now let $a, b \leq 2^n$. By the same reasoning, we can assume that $v_2(b) \geq 1$ and $v_2(a)= 0$. \\

\noindent Let $a' = (a + 1)/2$ and $b' = b/2$. Since $b - 2^{v_2(b)}$ is even and $a$ is odd, $a - b \leq 2^{v_2(b)} \Longleftrightarrow a \leq b + 2^{v_2(b)}$ implies that $a+1 \leq b + 2^{v_2(b)}$. Thus, $$a' \geq b' - 2^{v_2(b) - 1} = a' + 2^{v_2(b')}$$

\noindent By $a - b \geq 2$, we again have $a' > b' + 1$. Moreover, $$a' - 1 = \frac{a}{2} - \frac{1}{2} \leq \frac{b}{2} + 2^{v_2(b)-1} - \frac{1}{2} < b' + 2^{v_2(b')}$$

\noindent By the inductive hypothesis and the inequalities above, and by the fact that $v_2(a') = v_2(a)-1$, we see that $$m_j(I_{a'+b'}) = \left\{ \begin{array}{lcl}
m_{j}(I_{a'}) + m_{j}(I_{b'}) & \mbox{if}
& j \geq n - v_2(b) + 1 \\ b' & \mbox{if} & j = n - v_2(b) \\ m_{j+1}(I_{a'}) + m_{j+1}(I_{b'}) & \mbox{if} & j \leq n - v_2(b) - 1
\end{array}\right.$$
$$m_j(I_{a'+b'+1}) = \left\{ \begin{array}{lcl}
m_{j}(I_{a'-1}) + m_{j}(I_{b'}) & \mbox{if}
& j \geq n - v_2(b) + 1 \\ b'+1 & \mbox{if} & j = n - v_2(b) \\ m_{j+1}(I_{a'-1}) + m_{j+1}(I_{b'}) & \mbox{if} & j \leq n - v_2(b) - 1
\end{array}\right.$$

\noindent By Lemma \ref{m of 2a and 2a+1}, we also have $$m_j(I_{a+b}) = m_j(I_{2((a'-1)+b')+1}) = \left\{ \begin{array}{lcl}
m_{j+1}(I_{a'+b'-1}) + m_{j+1}(I_{a'+b'}) & \mbox{if}
& j \leq n - 1 \\ a'+b'-1 & \mbox{if} & j = n
\end{array}\right.$$

\noindent The rest of the proof follows by the same reasoning as Case 1.
\end{proof}

\begin{lem}
\label{v2(a+b)}
If $v_2(a) < v_2(b)$, then $v_2(a+b) = v_2(a), v_2(b)$. If $a > b$, then $v_2(a-b) = v_2(a)$.
\end{lem}
\begin{proof}
Since $2^{v_2(a)}$ divides $a$ and $b$, $2^{v_2(a)}$ divides $a+b$. Since $v_2(a) < v_2(b)$, $2^{v_2(a)+1}$ divides $b$. If $2^{v_2(a)+1}$ also divides $a+b$, then $2^{v_2(a)+1}$ would divide $a$, which contradicts the fact that $2^{v_2(a)}$ is the largest power of $2$ that divides $a$. Hence, $2^{v_2(a)+1}$ does not divide $a+b$, so $v_2(a+b) = v_2(a)$. For the same reason, if $a > b$, then $v_2(a-b) = v_2(a)$.
\end{proof}

\begin{prop}
\label{lower bound for fit pairs} Let $x_n$ be the number of pairs $(a, b)$ in $S_n = \{(a, b): 2^n - 2^{n-2} \leq a, b < 2^n\}$ that satisfy $|a - b| \leq 2^{\max \{v_2(a), v_2(b)\}}$. Then 
$$x_{n} \geq 2 x_{n-1} + 2^{n-1} - 2.$$
\end{prop}
\begin{proof}
Let $$T_1 = \{(a,b): 2^n-2^{n-2} \leq a, b < 2^n-2^{n-3}\}$$ $$T_2 = \{(a,b): 2^n-2^{n-2} \leq a < 2^n-2^{n-3}, 2^n-2^{n-3} \leq b < 2^n\}$$ $$T_3 = \{(a,b): 2^n-2^{n-3} \leq a < 2^n, 2^n-2^{n-2} \leq b < 2^n-2^{n-3}\}$$ $$T_4 = \{(a,b): 2^n-2^{n-3} \leq a, b < 2^n\}$$
Note that $S_n$ is the disjoint union of $T_1, T_2, T_3, T_4$, and $|S_{n-1}| = |T_1| = |T_2| = |T_3| = |T_4|$. \\

\noindent We first show that the number of pairs that satisfy the inequality in $T_1$ and $T_4$ is at least $x_{n-1}$. If $a = 2^n - 2^{n-2}$, then $v_2(a) = n-2$. If $2^n-2^{n-2} \leq b < 2^n - 2^{n-3}$, then $|a - b| \leq 2^{\max \{v_2(a), v_2(b)\}}$. Similarly, if $b = 2^n - 2^{n-2}$ and $2^n-2^{n-2} \leq a < 2^n - 2^{n-3}$, then $|a - b| \leq 2^{\max \{v_2(a), v_2(b)\}}$. Now if $2^n - 2^{n-2} < a, b < 2^n - 2^{n-3}$, then $v_2(a), v_2(b) < n-3 = v_2(2^{n-1} - 2^{n-3})$, so by the previous lemma, $v_2(a) = v_2(a - (2^{n-1} - 2^{n-3})), v_2(b) = v_2(b - (2^{n-1} - 2^{n-3}))$. By our bounds on $a$ and $b$, $(a - (2^{n-1} - 2^{n-3}), b - (2^{n-1} - 2^{n-3})) \in S_{n-1}$, so if $|(a - (2^{n-1} - 2^{n-3})) - (b - (2^{n-1} - 2^{n-3}))| = |a-b| \leq 2^{\max \{v_2(a - (2^{n-1} - 2^{n-3})), v_2(b - (2^{n-1} - 2^{n-3}))}$, then $|a-b| \leq 2^{\max \{v_2(a), v_2(b)\}}$. Hence, the number of pairs in $T_1$ that satisfy the inequality is at least those in $S_{n-1}$, which is $x_{n-1}$. Similarly, such number in $T_4$ is at least $x_{n-1}$. \\

\noindent Next, note that in $T_2$, the pairs $(2^n - 2^{n-2}, b)$ and $(a, 2^n-2^{n-3})$ satisfies the inequality; in $T_3$, the pairs $(2^n - 2^{n-3}, b)$ and $(a, 2^n-2^{n-2})$ satisfies the inequality. These count $2^{n-1} - 2$ number of pairs. Hence $x_n \geq 2x_{n-1} + 2^{n-1} - 2$.
\end{proof}

\noindent Next, we'll give a lower bound on the number of unfit pairs.

\begin{lem}
\label{column-0 unfit}
Let $$S'_n = \{(a, b): 2^n - 2^{n-2} < a, b < 2^n, |a - b| \geq 2\}.$$
If $(a, b) \in S'_n$ and $(a, b)$ fits, then there is an index $j$ such that $$m_j(I_{a+b}) = \min \{a, b\}.$$
\end{lem}
\begin{proof}
Without loss of generality, assume that $a > b$, so that $a-b \geq 2$. By our bounds on $(a, b) \in S'_n$, the m-table of $I_a$ and $I_b$ are
$$\left( \begin{array}{ccc} 0 & * & \cdots \\
a & * & \cdots \end{array} \right) \text{ and } \left( \begin{array}{ccc} 0 & b - 2^{n-1} & \cdots \\
b & 2^{n-1} & \cdots \end{array} \right).$$
Since $$\lfloor \frac{a+b}{2} \rfloor < \lfloor \frac{a+b}{2} \rfloor + 1 = \lfloor \frac{a+b+2}{2} \rfloor \leq \lfloor \frac{a+a}{2} \rfloor = a,$$
we have that $m_j(I_{a+b}) \leq \lfloor \frac{a+b}{2} \rfloor < a$ for each $j$. Hence, $0$ is the labeled entry in column 0 of the m-table of $I_a$. Moreover, by our bounds on $a$ and $b$, $$m_j(I_{a+b}) \geq m_0(I_{a+b}) = a+b-2^n > 2\cdot (2^n - 2^{n-2}) - 2^n = 2^{n-1}$$
Thus, the entry in the m-table of $I_b$ that pairs with $0$ can only be $b$. If this label is $j$, then $m_j(I_{a+b}) = b$.
\end{proof}

\begin{defn} If $m_j(I_{a+b}) \neq \min \{a, b\}$ for each index $j$, then call the pair $(a, b)$ to be column-0 unfit.
\end{defn}

\noindent By the above lemma, if a pair $(a, b) \in S'_n$ is column-0 unfit, then it is unfit. Additionally, when $(a, b) \in S'_n$ is column-0 unfit, the number $\min \{a, b\}$ lies between $m_j(I_{a+b})$ and $m_{j+1}(I_{a+b})$ for some index $j$. To see this, we assume again that $a > b$. Since $a < 2^n$, $m_0(I_{a+b}) = a + b - 2^n < b$. Since $a - b \geq 2$, $$m_n(I_{a+b}) = \lfloor \frac{a+b}{2} \rfloor > \lfloor \frac{a-2+b}{2} \rfloor \geq \lfloor \frac{b+b}{2} \rfloor = b.$$
Thus, $b$ is in between $m_0(I_{a+b})$ and $m_n(I_{a+b})$. Since $b$ is not equal to any $m_j(I_{a+b})$, $b$ must be in between $m_j(I_{a+b})$ and $m_{j+1}(I_{a+b})$ for some index $j$. \\

\noindent Now we construct a sequence of sets $A_n$, starting with $A_4 = \{(13, 15), (15, 13)\}$. For $n \geq 4$, let $A_{n+1}$ consist of: 
\begin{enumerate}
\item All elements in the set $\{2a_0-1, 2a_0, 2a_0+1\} \times \{2b_0-1, 2b_0, 2b_0+1\}$ for each $(a_0, b_0) \in A_n$.
\item $(a, a-2), (a-2, a) \in S'_{n+1}$ such that $a \equiv 3 \mod 4$.
\end{enumerate}

\noindent It is easily verifiable by induction that $A_n \subseteq S'_n$. 

\begin{prop}
\label{A_n column-0 unfit} For each $n \geq 4$, every element in $A_n$ is column-0 unfit.
\end{prop}
\begin{proof}
We prove by induction. When $n = 4$, it can be quickly verified by writing out the m-table that $(13, 15)$ and $(15, 13)$ is column-0 unfit. Now assume that every element in $A_{k-1}$ is column-0 unfit, $k > 4$. We'll prove the same for $A_{k}$. \\

\noindent Let $\alpha, \beta \in \{-1, 0, 1\}$, $(a_0, b_0) \in A_{k-1}$, and without loss of generality, $a_0 > b_0$, so that $a_0 - b_0 \geq 2$. By the remark above and the assumption that $(a_0, b_0)$ is column-0 unfit, we can find $j$ such that $m_j(I_{a_0 + b_0}) < b_0 < m_{j+1}(I_{a_0+b_0})$, which implies that $$2m_j(I_{a_0+b_0}) \leq 2b_0 - 2$$ $$2m_{j+1}(I_{a_0+b_0}) \geq 2b_0 + 2$$
We claim that $m_{j-1}(I_{2a_0+2b_0+\alpha+\beta}) < 2b_0 + \beta < m_j(I_{2a_0+2b_0+\alpha+\beta})$. Observe that
$$m_{j-1}(I_{2a_0+2b_0+\alpha+\beta}) \leq m_{j-1}(I_{2a_0+2b_0}) + \max \{\alpha + \beta, 0\} = 2m_{j}(I_{a_0+b_0}) + \max \{\alpha + \beta, 0\} \leq 2b_0-2+ \max \{\alpha + \beta, 0\},$$
and
$$m_j(I_{2a_0+2b_0+\alpha+\beta}) \geq m_{j}(I_{2a_0+2b_0}) + \min \{\alpha + \beta, 0\} = 2m_{j+1}(I_{a_0+b_0}) + \min \{\alpha + \beta, 0\} \geq 2b_0+2+ \min \{\alpha + \beta, 0\}.$$
The claim immediately follows from the observation that $\max \{\alpha + \beta, 0\} - 2 < \beta$ and $\min \{\alpha + \beta, 0\} + 2 > \beta$. Moreover, by lemma \ref{orderinglemma}, $(2a_0 + \alpha, 2b_0 + \beta)$ is column-0 unfit. \\

\noindent We're left to prove that $(a, a-2), (a-2, a) \in S'_{k}$ such that $a \equiv 3 \mod 4$ is column-0 unfit. By symmetry we only need the consider $(a, a-2)$. Since $a \equiv 3 \mod 4$, $2a -2 = a + (a-2) \equiv 0 \mod 4$, which means that $m_k(I_{2a-2}) = m_{k-1}(I_{2a-2}) = (2a-2)/2 = a-1$. At the same time, $2a-2$ is always $4 \mod 8$ when $a \equiv 3 \mod 4$, which means that $m_{k-2}(I_{2a-2}) \equiv 0 \mod 4$. Since $a-1 \equiv 2 \mod 4$ and $m_{k-1}(I_{2a-2}) \geq m_{k-2}(I_{2a-2})$, we have that $$m_{k-2}(I_{2a-2}) \leq m_{k-1}(I_{2a-2}) -2 = a-3.$$
Thus, $m_{k-2}(I_{2a-2}) < a-2 < m_{k-1}(I_{2a-2})$, which implies that $(a, a-2)$ is column-0 unfit.
\end{proof}

\noindent Hence, counting the number of elements in $A_n$ will give us a lower bound for the number of unfit pairs in $S_n$. 

\begin{defn}
Let $A, B \subseteq \Z^2$. If $$B \cap \bigcup_{(a, b) \in A} \{a-1, a, a+1\} \times \{b-1, b, b+1\} = \emptyset$$
or
$$A \cap \bigcup_{(c,d)\in B} \{c-1,c,c+1\} \times \{d-1,d,d+1\} = \emptyset$$
Then $A$ and $B$ are said to be separated.
\end{defn} 

\noindent It is easy to see that the two conditions are equivalent.

\begin{lem}
\label{separated after multiplying} 
If $A, B \in \Z^2$ are separated, then 
$$A' = \bigcup_{(a, b) \in A} \{2a-1, 2a, 2a+1\} \times \{2b-1, 2b, 2b+1\}$$
and
$$B' = \bigcup_{(c,d)\in B} \{2c-1,2c,2c+1\} \times \{2d-1,2d,2d+1\}$$
are separated.
\end{lem}
\begin{proof}
Suppose not. Let $(a', b') \in A'$ and $(a' + \alpha, b' + \beta) = (c', d') \in B'$, where $\alpha, \beta \in \{-1, 0, 1\}$. By definition, write $(a', b') = (2a + \alpha_1, 2b + \beta_1)$ and $(c',d') = (2c + \alpha_2, 2d + \beta_2)$, where $\alpha_1, \beta_1, \alpha_2, \beta_2 \in \{-1, 0, 1\}$. Then $(2a + \alpha_1 + \alpha, 2b + \beta_1 + \beta) = (2c + \alpha_2, 2d+\beta_2)$, which implies that $$(2a + \alpha + \alpha_1 - \alpha_2, 2b + \beta + \beta_1 - \beta_2) = (2c+2d).$$ 
Modding out by 2, we see that $\alpha + \alpha_1 - \alpha_2, \beta + \beta_1 - \beta_2 \equiv 0 \mod 2$, so they must be $-2, 0$ or $2$. Let $\alpha_0 = (\alpha + \alpha_1 - \alpha_2)/2$ and $\beta_0 = (\beta + \beta_1 - \beta_2)/2$. Then $\alpha_0, \beta_0 \in \{-1, 0, 1\}$, and $(a + \alpha_0, b + \beta_0) = (c,d)$ by the above equality. This contradicts the assumption that $A, B$ are separated.
\end{proof}

\begin{prop}
\label{size of A_n} $|A_n| = 2|A_{n-1}| + 2^{2n-5} - 2^{n-1} + 2$. Thus, if $y_n$ be the number of unfit pairs $(a, b)$ in $S_n = \{(a, b): 2^n - 2^{n-2} \leq a, b < 2^n\}$, then
$$y_{n} \geq 2y_{n-1} + 2^{2n-5} - 2^{n-1} + 2.$$
\end{prop}
\begin{proof}
We claim that $A_n$ consists of $2^{n-2-j}$ copies of $2^j - 1$ by $2^j-1$ square for all $j$ that satisfies $1 \leq j \leq n-3$. Moreover, each square is separated from all other squares. \\

\noindent The case $n = 4$ can be easily checked. Suppose the case $n = k-1$ is known, where $k > 4$, and we wish to prove the same for $A_k$. First, if $(a, a-2) \in A_k$ and $a \equiv 3 \mod 4$, it can be checked directly that, if $(c, d)$ is any one of the surrounding pair, $|c - d| \leq 2^{\max \{v_2(c), v_2(d)\}}$, which implies that $(c,d)$ fits. But each element in $A_k$ is unfit, so $(c,d) \not\in A_k$. Thus, $(a, a-2)$ is separated from all other elements of $A_k$. This accounts for the $2^{k-3}$ copies of $1$ by $1$ square. \\

\noindent Suppose $B_1, B_2, \ldots, B_N$ are the squares in $A_{k-1}$, each separated from all others. By construction, $B'_1, B'_2, \ldots, B'_N$ are the remaining squares in $A_k$, where 
$$B'_i = \bigcup_{(a, b) \in A_i} \{2a-1, 2a, 2a+1\} \times \{2b-1, 2b, 2b+1\}$$
If $B_i$ is a $2^j-1$ by $2^j-1$ square, where $1 \leq j \leq k-4$ then $B'_i$ is a $2^{j+1}-1$ by $2^{j+1}-1$ square. Since $B_i$ and $B_j$ are separated if $i \neq j$, $B'_i$ and $B'_j$ are separated by the above lemma. Moreover, each $B'_i$ is separated from the $1$ by $1$ squares, since we've shown that they are separated from all other elements of $A_k$. Finally, since there are $2^{k-3-j}$ copies of $2^j -1$ by $2^j -1$ squares in $A_{k-1}$, where $1 \leq j \leq k-4$, these squares generate $2^{k-2-j}$ copies of $2^j-1$ by $2^j-1$ squares in $A_k$, where $2 \leq j \leq k-3$. This proves the claim. \\

\noindent To count the size of $A_n$ with respect to the size of $A_{n-1}$, note that $A_{n-1}$ has $2^{n-3-j}$ copies of $2^j-1$ by $2^j-1$ squares for all $1 \leq j \leq n-4$, while $A_n$ has $2^{n-2-j}$ copies of $2^j-1$ by $2^j-1$ squares for all $1 \leq j \leq n-4$ and $2$ copies of $2^{n-3}-1$ by $2^{n-3}-1$ squares. Thus, 
$$|A_n| = 2|A_{n-1}| + 2 \cdot (2^{n-3}-1)^2 = 2|A_{n-1}| + 2^{2n-5} - 2^{n-1} + 2.$$
\end{proof}

\begin{prop}
\label{Remaining Cases}
The number of unfit pairs in the set $\{(a, b): 2^n - 2^{n-2} \leq a, b < 2^n\}$ is 
$$4^{n-2} - (n-2)2^{n-1} + 2^{n-2} - 2.$$
And they are precisely the set 
$$\{(a, b): 2^n - 2^{n-2} < b < a < 2^n, |a - b| > 2^{\max \{v_2(a), v_2(b)\}}\}$$
\end{prop}
\begin{proof}
\noindent By our numerical results, $x_4 = 14$ and $y_4 = 2$, so $x_4 + y_4 = 4^2$. We shall prove by induction that $x_n + y_n = 4^{n-2}$ for each $n \geq 4$. Suppose for induction that this is true for $n-1$, $n \geq 5$. Then
$$x_n + y_n \geq (2x_{n-1} + 2^{n-1} - 2) + (2y_{n-1} + 2^{2n-5} - 2^{n-1} + 2) = 2(x_{n-1} + y_{n-1}) + 2^{2n-5} = 2 \cdot 4^{n-3} + 2^{2n-5} = 4^{n-2}$$
\noindent Since $x_n, y_n$ counts the number of points in two disjoint subsets of $S_n$, and $|S_n| = 4^{n-2}$, it follows that $x_n + y_n = 4^{n-2}$. Moreover, the bounds given by Proposition \ref{lower bound for fit pairs} and Proposition \ref{size of A_n} must be equality. This immediately gives the second part of the Proposition.

\noindent In particular, for each $n \geq 5$, $$y_n = 2y_{n-1} + 2^{2n-5} - 2^{n-1} + 2$$

\noindent It is straightforward to verify that $y_n = 4^{n-2} - (n-2)2^{n-1} + 2^{n-2} - 2$ is a closed form for the recursion. Thus, the number of unfit pairs in the set $\{(a, b): 2^n - 2^{n-2} \leq b < a < 2^n\}$ is $2^{2n-5} - (n-2)2^{n-2} + 2^{n-3} - 1.$ 
\end{proof} 
\vspace{10mm}

\section{Number of Unfit Pairs}


\noindent In this section, we'll derive a closed formula for the number of unfit pairs $(a, b)$ with $(a, b) \leq 2^n$.

\begin{prop}
\label{unfitrecurrence}
Let $t_n$ be the number of unfit pairs with $(a, b) \leq 2^n$. Then $t_4 = 24$ and $t_n$ satisfies the following recurrence for $n \geq 5$:
$$t_n = 2t_{n -1} + \sum_{j = 4}^{n - 2} 2^{n - 2 - j} t_{j} + 6\left(4^{n - 2} - (n - 1)2^{n - 2}\right).$$
\end{prop}

\begin{proof}
Let $u_n$ denote the number of unfit pairs in the upper half $\{(a, b): a, b > 0, a+b \leq 2^n\},$ and let $l_n$ denote the number of unfit pairs in the lower half $\{(a, b): a, b \leq 2^n, a+b > 2^n\}$. Clearly, $t_n = u_n + l_n$, and $l_n$ is also equal to the number of unfit pairs in $\{(a, b): a, b < 2^n, a+b > 2^n\}$ or in $\{(a,b):a, b \leq 2^n, a+b \geq 2^n\}$, since $(2^n, b), (a, 2^n),$ and $(a, 2^n-a)$ are all fit pairs for each $0<a, b \leq 2^n$. \\

\noindent We begin by establishing a recurrence for $u_n$. Observe that $u_n$ is the disjoint union of 
$$\{(a,b): 0 < a, b \leq 2^{n-1}\} \cup \{(a, b): a>2^{n-1}, b>0, a+b \leq 2^n\} \cup \{(a,b):a>0, b>2^{n-1}, a+b\leq 2^n\}.$$
By Proposition \ref{Translation of the Upper Triangle}, the number of unfit pairs in each of the latter two sets is $u_{n-1}$. By definition, the number of unfit pairs in the first set is $t_{n-1}$. Thus, 
$$u_n = t_{n-1} + 2u_{n-1}$$

\noindent Now we count the unfit pairs given in section 4.3 and 4.4. Observe that the set $S_1 = \{(a,b): a, b < 2^n, a+b > 2^n - 2^{n-1}\}$ is the disjoint union of
$$\{(a,b):a<2^n, b<2^n-2^{n-2}, a+b > 2^n+2^{n-1}\} \cup \{(a,b): a<2^n-2^{n-2}, b<2^n, a+b>2^n+2^{n-1}\}$$ $$\cup \{(a,b): 2^n-2^{n-2} \leq a, b < 2^n\}.$$
By Proposition \ref{New Large Triangle}, the first two set each counts $(4^{n-2} - 3\cdot 2^{n-2}+2)/2$ unfit pairs, and by Proposition \ref{Remaining Cases}, the last set counts $4^{n-2} - (n-2)2^{n-1} + 2^{n-2}-2$ unfit pairs. Thus, the number of unfit pairs in $S_1$ is $$(4^{n-2} - (n-2)2^{n-1} + 2^{n-2}-2) + (4^{n-2} - 3\cdot 2^{n-2}+2) = 2 \cdot 4^{n-2} - (n-1)2^{n-1}.$$
Next, consider the mapping $(a, b) \rightarrow (a, 2^{n+1} - a - b)$, which maps fit pairs to fit pairs and unfit pairs to unfit pairs. The image $S_2 = \{(a, b): a< 2^n, b<2^{n-1}, a+b > 2^n+2^{n-1}\}$. Consider also $S_3 = \{(a, b): a < 2^{n-1}, b<2^n, a+b > 2^n+2^{n-1}\}$, which is the mirror symmetry of $S_2$. Note that $S_1, S_2,$ and $S_3$ are pairwise disjoint and contained in the lower half, and they have the same number of unfit pairs. The subset of the lower half that has not been counted is the set
$$\{(a,b): a \geq 2^{n-1}, b \geq 2^{n-1}, a+b \leq 2^n+2^{n-1}\}.$$
But by Proposition \ref{Reflection of the Lower Triangle}, the number of unfit pairs in this set is precisely $l_{n-1}$. Putting these all together, we have
$$t_n = u_n+l_n = t_{n-1} + 2u_{n-1} + l_{n-1} + 3(2 \cdot 4^{n-2} - (n-1)2^{n-1}) = 2t_{n-1} + u_{n-1} + 6(4^{n - 2} - (n - 1)2^{n - 2}).$$
The recurrence follows immediately by expanding $u_{n-1}$ recursively, terminating at $u_4=0$.

\noindent 
\end{proof}

\begin{thm}
For all $n \geq 4$, 
$$t_n = 4^n - \binom{4}{3}3^n + \binom{4}{2}2^n - \binom{4}{1},$$
where $t_n$ is the number of unfit pairs $(a, b)$ with $0 < a, b \leq 2^n$ as above.
\end{thm}

\begin{proof}
It is easy to verify that $s_4 = t_4 = 24$. For the rest of the sequence, we induct on $n$. Assume that for for an arbitrary $k \geq 4$, the claim holds for all $4 \leq j \leq k$, i.e.
$$t_j = 4^j - \binom{4}{3}3^j + \binom{4}{2}2^j - \binom{4}{1}.$$

\noindent Then we have
\begin{align*}
t_{k + 1} &= 2t_k + \sum_{j = 4}^{k - 1} 2^{k - 1 - j} t_{j} + 6\left(4^{k - 1} - k\cdot2^{k - 1}\right)\\
&= 2 \left( 4^k - \binom{4}{3}3^k + \binom{4}{2}2^k - \binom{4}{1} \right) + \sum_{j = 4}^{k - 1} 2^{k - 1 - j}\left(4^j - \binom{4}{3}3^j + \binom{4}{2}2^j - \binom{4}{1}\right) + 6\left(4^{k - 1} - k\cdot2^{k - 1}\right)\\
\end{align*}
The sum $t_{k+1}$ above can be decomposed into four parts $S_1 + S_2 + S_3 + S_4$ as follow:

$$S_1 = 2 \cdot 4^k + \sum_{j = 4}^{k - 1} 2^{k - 1 - j} \cdot 4^j + 6 \cdot 4^{k-1} = 4^{k+1} - 2^{k+3};$$

$$S_2 = - \binom{4}{3} \left( 2 \cdot 3^k + \sum_{j = 4}^{k - 1} 2^{k - 1 - j} \cdot 3^j \right) = - \binom{4}{3} 3^{k+1} + 81 \cdot 2^{k-2};$$

$$S_3 = \binom{4}{2} \left( 2 \cdot 2^k + \sum_{j = 4}^{k - 1} 2^{k - 1 - j} \cdot 2^j \right) - 6k \cdot 2^{k-1} = \binom{4}{2} 2^{k+1} - 3 \cdot 2^{k+2};$$

$$S_4 = -\binom{4}{1} \left( 2 + \sum_{j = 4}^{k - 1} 2^{k - 1 - j} \right) = - \binom{4}{1} - 2^{k-2},$$

\noindent where all the second equalities are obtained by simplifying the expression. Since 

$$-2^{k+3} + 81 \cdot 2^{k-2} - 3 \cdot 2^{k+2} - 2^{k-2} = 2^{k-2} (81 - 32 - 48 - 1) = 0, $$

\noindent we have 

$$t_{k+1} = S_1 + S_2 + S_3 + S_4 = 4^{k+1} - \binom{4}{3}3^{k+1} + \binom{4}{2}2^{k+1} - \binom{4}{1}.$$
\end{proof}

\noindent The closed formula for the number of unfit pairs lines up exactly with \cite{OEIS919} which describes the sequence of numbers produced by the number of surjections from an n-element set onto a four-element set, a sequence defined by $S(n,4)4!$, where $S(n,4)$ corresponds to the Stirling Number of the Second Kind at $(n,4)$ in \cite{stirling}. 

\begin{cor} 
The density of unfit pairs as $n \to \infty$ approaches $1$.
\end{cor}
\vspace{1em}
\section*{Acknowledgments}

\noindent Special thanks to Professor Peter Johnson for running the Auburn University Summer REU, in which the work for the current paper was conduced, and to Professor Joe Briggs who inspired us to take on this problem.

\end{document}